\input ./style/arxiv-general.cfg
\documentclass[MSNbibl,number,citesort,seceqn,secthm,dvips]{arxbj}
\makeatletter
   \@ifpackageloaded{graphicx}{}{\usepackage{graphicx}}
\makeatother

\volume{23}
\issue{1}
\pubyear{2017}
\firstpage{219}
\lastpage{248}
\doi{10.3150/15-BEJ742}
\docsubty{FLA}

\makeatletter
\newcommand{\rrVert}{\Vert}
\newcommand{\llVert}{\Vert}
\newtheorem{lemma}{Lemma}[section]
\newproclaim{mydef}{Definition}[section]
\newtheorem{cor}{Corollary}[section]
\newremark{remark}{Remark}[section]
\newtheorem{prop}{Proposition}[section]
\newcommand{\eqref}[1]{(\ref{#1})}
\makeatother

\begin{document}
\begin{frontmatter}

\title{Optimal exponential bounds for aggregation of density estimators}
\runtitle{Optimal exponential bounds for aggregation of density estimators}

\begin{aug}
\author[A,B]{\inits{P.C.}\fnms{Pierre C.}~\snm{Bellec}\corref{}\thanksref{A,B}\ead
[label=e1]{pierre.bellec@ensae.fr}}
\address[A]{CREST-ENSAE, 3 avenue Pierre Larousse, 92245
Malakoff Cedex, France.}
\address[B]{CMAP, Ecole Polytechnique, Route de Saclay, 91120
Palaiseau, France.\\ \printead{e1}}
\end{aug}

%
\received{\smonth{1} \syear{2015}}
%
\revised{\smonth{4} \syear{2015}}

\begin{abstract}
We consider the problem of model selection type aggregation
in the context of density estimation.
We first show that empirical risk minimization is sub-optimal
for this problem and it shares this property with
the exponential weights aggregate,
empirical risk minimization over the convex hull of the dictionary functions,
and all selectors.
Using a penalty inspired by recent works on the
$Q$-aggregation procedure, we derive a sharp oracle inequality
in deviation
under a simple boundedness
assumption and we show that the rate is optimal in a minimax sense.
Unlike the procedures based on exponential weights,
this estimator is fully adaptive under the uniform prior.
In particular, its construction
does not rely on the sup-norm of the unknown density.
By providing lower bounds with exponential tails,
we show that the deviation term appearing in the sharp oracle
inequalities cannot be improved.
\end{abstract}

\end{frontmatter}

\section{Introduction}\label{introduction}
We study the problem of estimation of
an unknown density from observations.
Let $(\mathcal{X},\mu)$ be a measurable space.
We are interested in estimating an unknown density $f$ with respect to
the measure $\mu$ given $n$ independent observations $X_1,\ldots,X_n$ drawn
from $f$.
We measure the quality of estimation of $f$ by
the $L^2$ squared distance
\begin{equation}
\label{eqL2distance}
\llVert \hat g - f \rrVert ^2 = \int(f-\hat
g)^2\,d\mu= \llVert \hat g \rrVert ^2 - 2 \int \hat gf\, d
\mu+ \llVert f \rrVert ^2,
\end{equation}
for any $\hat g\in L^2(\mu)$
possibly dependent on the data $X_1,\ldots,X_n$.
Since the term $\llVert  f \rrVert ^2$ is constant for all
$\hat g$, we will
consider throughout the paper
the risk
\begin{equation}
\label{eqlossR}
R(\hat g) = \llVert \hat g \rrVert ^2 - 2 \int\hat g f
\,d\mu.
\end{equation}
An estimator $\hat g$ minimizes $R(\cdot)$ if and only if it minimizes
(\ref{eqL2distance}).

Given $M$ functions $f_1,\ldots,f_M\in L^2(\mu)$,
we would like to construct a measurable function $\hat g$
of the observations $X_1,\ldots,X_n$ that is almost as good
as the best function among $f_1,\ldots,f_M$.
The model may be misspecified, which means that $f$ may not be
one of the functions $f_1,\ldots,f_M$.
We are interested in
deriving oracle inequalities, either in expectation
\begin{eqnarray*}
&&\mathbb{E}R(\hat g) \le C \min_{j=1,\ldots,M} R(f_j) +
\delta_{n,M},
\end{eqnarray*}
or with high probability, that is, for all $\varepsilon> 0$, with
probability greater than $1-\varepsilon$
\begin{eqnarray*}
&& R(\hat g) \le C \min_{j=1,\ldots,M} R(f_j) +
\delta_{n,M} + d_{n,M}(\varepsilon),
\end{eqnarray*}
where $\delta_{n,M}$ is a small quantity and $d_{n,M}(\cdot)$ is a
function of $\varepsilon$
that we call the deviation term.
We are only interested in sharp oracle inequalities,
that is, oracle inequalities where the leading constant is $C=1$,
since it is essential to derive minimax optimality results.

We consider only deterministic functions for $f_1,\ldots,f_M$.
They cannot depend on the data $X_1,\ldots,X_n$.
A standard application of this setting was introduced in
Wegkamp
\cite{wegkamp1999quasi}: given $m+n$ i.i.d. observations drawn from $f$,
use the first $m$ observations to build $M$ estimators $\hat f_1,\ldots, \hat f_M$,
and in a second step use the remaining $n$ observations
to select
the best among the preliminary estimators $\hat f_1,\ldots, \hat f_M$.
A related problem is selecting the best estimator from a family $\hat
f_1,\ldots, \hat f_M$ where these estimators are built using the same
data used for model selection or aggregation.
Such problems were recently considered in
Dalalyan and Salmon
\cite{dalalyan2012sharp}
and
Dai \textit{et al.}
\cite{dai2014aggregation} for the regression model with fixed design.

We are also interested in deriving
sharp oracle inequalities with prior weights
on the model $\{f_1,\ldots,f_M\}$.
To be more precise, for some prior probability
distribution $\pi_1,\ldots,\pi_M$ over the finite set
$\{f_1,\ldots,f_M\}$ and any $\varepsilon> 0$, our estimator
$\hat f_n$ should satisfy with probability greater than $1-\varepsilon$
\begin{equation}
\label{eqoracle-ineq}
R(\hat f_n) \le \min_{j=1,\ldots,M} \biggl(
R(f_j) + \frac{\beta}{n} \log\frac{1}{\pi_j} \biggr) +
d_{n,M}(\varepsilon),
\end{equation}
for some positive constant $\beta$ and some deviation term
$d_{n,M}(\cdot)$.
The Mirror Averaging algorithm
\cite{juditsky2008learning,dalalyan2012mirror} is known
to achieve a similar oracle inequality in expectation.
The analysis of Juditsky et~al. \cite{juditsky2008learning}
shows that the constant $\beta$ scales linearly with the sup-norm
of the unknown density,
which is also the case for the results presented here.
Model selection techniques with prior weights
were used in order to derive
sparsity oracle inequalities
using sparsity pattern aggregation
\cite{rigollet2011exponential,rigollet2012sparse,dalalyan2012mirror}.

Another related learning problem is that
of model selection when the model is
finite dimensional with a specific shape,
for example a linear span of $M$ functions
or the convex hull of $M$ functions.
This is the aggregation framework
and it has received a lot of attention in the last decade
to construct adaptive estimators that achieve the minimax optimal rates,
especially for the regression problem
\cite
{tsybakov2003optimal,lounici2007generalized,rigollet2011exponential,lecue2013empirical,rigollet2012sparse}
but also for density estimation
\cite{yang2000mixing,lecue2006lower,rigollet2007linear}.

The main contribution of the present paper is the following.
\begin{itemize}
\item
We provide sharp oracle inequalities
and the corresponding tight lower bounds
for two procedures: empirical risk minimization
over the discrete set $\{f_1,\ldots,f_M\}$
and the penalized procedure \eqref{eqdef-that} with
the penalty \eqref{eqdef-pen}.
Here, tight means that neither the rate nor the deviation
term of the sharp oracle inequalities can be improved.
The sharp oracle inequalities are given in
Theorem~\ref{thmerm-oi} and Corollary~\ref{coruniform}
and the tight lower bounds are given in
Theorems~\ref{thmselectors}
and~\ref{thmlower-exp}.
These results lead to a definition
of minimax optimality in deviation, which is discussed in
Section~\ref{sminimax}.
\end{itemize}
While proving the above results,
we extend several aggregation results that are known
for the regression model to the density
estimation setting.
Let us relate these results of the present paper to the existing
literature on the regression model:
\begin{itemize}
\item
In Theorem~\ref{thmerm-oi}, we derive a sharp oracle inequality in deviation
for the empirical
risk minimizer over the discrete set $\{f_1,\ldots,f_M\}$.
This is new in the context of density estimation,
and an analogous result is known
for the regression model \cite{rigollet2012sparse}. 
\item
In Theorem~\ref{thmms-oi}, we derive a sharp oracle inequality in deviation
for penalized empirical risk minimization with the penalty (\ref{eqdef-pen}).
With the uniform prior,
this yields the correct rate $(\log M)/n$ of model selection type aggregation.
This penalty is inspired by recent works on the $Q$-aggregation
procedure \cite{lecue2014optimal,dai2012deviation}
where similar oracle inequalities in deviation were obtained for the
regression model.
The first sharp oracle inequalities
that achieve the correct rate of model selection type aggregation were obtained
in expectation for the regression model
in \cite{yang2000mixing,catoni2004statistical}.
\item
We extend several lower bounds known for the regression model
to the density estimation setting.
We show that any procedure that selects a dictionary function
cannot achieve a better rate than $\sqrt{(\log M) / n}$
and that the rate of model selection type aggregation is of order
$(\log M)/n$.
We also show that the exponential weights aggregate and the
empirical risk minimizer over the convex hull of the dictionary
functions cannot be optimal in deviation,
with an unavoidable error term of order $1/\sqrt{n}$.
Earlier results for the regression model can be found in
\cite{tsybakov2003optimal,rigollet2012sparse} for lower bounds on
model selection type aggregation and the performance of selectors,
while \cite{lecue2009aggregation,dai2012deviation,lecue2013optimality}
contain earlier lower bounds on the performance of exponential weights
and empirical risk minimization over the convex hull of the dictionary.
\end{itemize}

An aspect of our results is not present
in the previous works on the regression model.
In the literature on aggregation in the regression model, lower bounds
are proved either in expectation or in probability in the form
\begin{equation}
\label{eqexample-lower-proba}
\mathbb{P} \Bigl( R(\hat T_n) > \min
_{j=1,\ldots,M} R(f_j) + \psi_{n,M} \Bigr) > c,
\end{equation}
for any estimator $\hat T_n$, a risk function $R(\cdot)$, a rate $\psi_{n,M}$
and some absolute constant $c>0$, usually $c=1/2$.
The tight lower bounds presented in Theorems~\ref{thmselectors}
and~\ref{thmlower-exp} contrast with
lower bounds of the form \eqref{eqexample-lower-proba} as they yield
for any estimator $\hat T_n$,
\begin{equation}
\label{eqexample-lower-range} \forall x > 0, \qquad\mathbb{P} \biggl( R(\hat T_n) >
\min_{j=1,\ldots,M} R(f_j) + \psi_{n,M} +
\frac{x}{n} \biggr) > c \exp(-x),
\end{equation}
that is, they provide lower bounds for any probability estimate in an
interval $(0,1/c)$
where $c>0$ is an absolute constant.
Moreover, these lower bounds show that the exponential tail of  the
excess risk of the estimators from Theorems~\ref{thmerm-oi}
and~\ref{thmms-oi} cannot be improved.
The tools used in the present paper to prove lower bounds of the form
\eqref{eqexample-lower-range}, in particular Lemma~\ref{lemmaminimax},
can be used to prove similar results for regression model.
The tight lower bounds of the present paper contrast
with the existing literature on the regression model,
since to our knowledge, there is no lower bound of the form \eqref
{eqexample-lower-range} available for regression.

In the regression model with random design, given a class of functions
$G$, a penalty $\operatorname{pen}( \cdot )$, a~coefficient $\nu>0$
and observations $(X_1,Y_1),\ldots,(X_n,Y_n)$,
penalized empirical risk minimization
solves the optimization problem
\begin{equation}
\label{eqexample-erm-regression}
\min_{g\in G}   \frac{1}{n} \sum
_{i=1}^n \bigl(g(X_i) -
Y_i\bigr)^2 + \nu\operatorname{pen}( g ).
\end{equation}
But if the distribution of the design is known,
the statistician can compute the quantity $\mathbb{E}[g(X)^2]$ for all
$g\in G$
and solve the following minimization problem that slightly differs from
\eqref{eqexample-erm-regression}:
\begin{equation}
\label{eqexample-erm-known} \min_{g\in G}   \mathbb{E}\bigl[g(X)^2
\bigr] - \frac{2}{n} \sum_{i=1}^n
g(X_i) Y_i + \nu \operatorname{pen}( g ).
\end{equation}
In the regression model, the distribution of the design is rarely known
so the
penalized ERM that solves \eqref{eqexample-erm-known}
has not received as much attention as the procedure \eqref
{eqexample-erm-regression}
when the distribution of the design is not known.
The density estimation setting studied in the present paper
is closer to the regression setting
with known design \eqref{eqexample-erm-known}
than to the regression setting with unknown design~\eqref
{eqexample-erm-regression}
studied in
\cite{lecue2014optimal}.
There are differences with respect
to the choice of coefficient of the penalty \eqref{eqdef-pen},
and to the form of the empirical process that appears in the analysis.
These differences are more thoroughly discussed in Section~\ref{sdiff-reg}.

The paper is organized as follows.
In Section~\ref{sselectors}, we show that
empirical risk minimization achieves a sharp oracle inequality
with slow rate, but this rate cannot be improved
among selectors. Two classical estimators,
the exponential weights aggregate and empirical risk minimization over
the convex hull of the dictionary functions, are shown to be suboptimal
in deviation.
In Section~\ref{sms},
we define a penalized procedure
that achieves the optimal rate $\frac{\log M}{n}$ in deviation,
and we provide a lower bound that shows that
neither the rate nor the deviation term can be improved.
Section~\ref{sminimax} proposes a definition of minimax optimality
in deviation and shows that it
is satisfied by the procedures given in Sections~\ref{sselectors} and  \ref{sms}.
Section~\ref{sproofs} is devoted to the proofs.

\section{Sub-optimality of selectors, ERM and exponential
weights}
\label{sselectors}

\subsection{Selectors}

Define a selector as a function of the form $f_{\hat J}$
where $\hat J$ is measurable with respect to $X_1,\ldots,X_n$
with values in $\{1,\ldots,M\}$.
It was shown in the regression framework
\cite{juditsky2008learning,rigollet2012sparse}
that selectors are suboptimal and cannot
achieve a better rate that $\sigma\sqrt{\frac{\log M}{n}}$ where
$\sigma^2$ is the variance of the regression noise.
The following theorem extends this lower bound for
selectors to density estimation.
The underlying measure $\mu$
is the Lebesgue measure on $\mathbf{R}^d$ for $d\ge1$.

\begin{thm}[(Lower bounds for selectors)]
\label{thmselectors}
Let $L>0$, and $M\ge2, n\ge1, d\ge1$ be integers.
Let $\mathcal{F}$ be the class of all densities $f$ with respect to
the Lebesgue measure on $\mathbf{R}^d$ such that $\Vert f\Vert_\infty
\le L$.
Let $x\ge0$ satisfying
\begin{eqnarray*}
&&\frac{\log(M) + x}{n} < 3.
\end{eqnarray*}
Then there exist $f_1, \ldots, f_M\in L^2(\mathbf{R}^d)$ with $\Vert
f_j\Vert_\infty\le L$ such that
the following lower bound holds:
\begin{eqnarray*}
&&\inf_{\hat S_n} \sup_{f\in\mathcal{F}} \mathbb{P}_f
\biggl( \Vert\hat S_n -f\Vert^2 - \inf
_{j=1,\ldots,M} \Vert f_j-f\Vert^2 \ge
\frac{L}{\sqrt{3}} \sqrt{\frac{x+ \log M}{n}} \biggr) \ge\frac{1}{24} \exp(-x),
\end{eqnarray*}
where
$\mathbb{P}_f$ denotes the probability with respect to $n$ i.i.d.
observations with density $f$
and the infimum is taken over all selectors $\hat S_n$.
\end{thm}

The proof of Theorem~\ref{thmselectors} is given in Section~\ref{sproofs}.
It can be extended to other measures as soon as
the underlying measurable space allows the construction
of an orthogonal system such as the one described in Proposition~\ref
{proprademacher} below.

For any $g\in L^2(\mu)$, define the empirical risk
\begin{equation}
\label{eqempiricalRn} R_n(g) = \llVert g \rrVert ^2 -
\frac{2}{n} \sum_{j=1}^M
g(X_i).
\end{equation}
The empirical risk (\ref{eqempiricalRn}) is an unbiased estimator of
the risk (\ref{eqlossR}).
In order to explain the idea behind the proof of our main result
described in Theorem~\ref{thmms-oi}, it is useful the prove the following
oracle inequality for the empirical risk minimizer over
the discrete set $\{f_1,\ldots,f_M\}$.
\begin{thm}
\label{thmerm-oi}
Assume that the functions $f_1, \ldots, f_M\in L^2(\mu)$ satisfy
$\llVert  f_j \rrVert _\infty\le L_0$ for all $j=1,\ldots,M$.
Define
\begin{eqnarray*}
&&\hat J \in\mathop{\operatorname{argmin}}_{j=1,\ldots,M} \Biggl( \llVert
f_j \rrVert ^2 - \frac{2}{n} \sum
_{i=1}^n f_j (X_i)
\Biggr).
\end{eqnarray*}
Then for any $x>0$, with probability greater than $1-\exp(-x)$,
\begin{eqnarray*}
R(f_{\hat J}) &\le&\min_{j=1,\ldots,M} R(f_j) +
L_0 \biggl( 4 \sqrt{2} \sqrt{\frac{x + \log M}{n}} +
\frac{8(x + \log M)}{3n} \biggr).
\end{eqnarray*}
\end{thm}
Together with Theorem~\ref{thmselectors}, Theorem~\ref{thmerm-oi} shows that empirical
risk minimization is optimal among selectors.
Unlike the oracle inequality of Theorem~\ref{thmms-oi} below,
this result applies for any density $f$,
with possibly $\llVert  f \rrVert _\infty = \infty$.
Its proof relies on the concentration
of $R_n(g) - R(g)$ around $0$ for fixed functions $g$ with $\llVert
g \rrVert _\infty\le L_0$.
\begin{pf*}{Proof of Theorem~\protect\ref{thmerm-oi}}
We will use the following notation that is common in the literature
on empirical processes. For any $g\in L^2(\mu)$, define
\begin{eqnarray}
P g & = &\int gf \,d\mu,
\nonumber
\\[-8pt]
\label{eqPPn1}
\\[-8pt]
\nonumber
P_n g & =& \frac{1}{n} \sum_{i=1}^n
g(X_i).
\end{eqnarray}
With this notation, the difference between the real risk (\ref
{eqlossR}) and the empirical risk (\ref{eqempiricalRn}) can be rewritten
\begin{equation}
\label{eqPPn}
R(g) - R_n(g) = (P-P_n) (-2g).
\end{equation}

Let $J^*$ be such that $R(f_{J^*}) = \min_{j=1,\ldots,M} R(f_j)$.
The definition of $\hat J$ yields
$R_n(f_{\hat J}) \le R_n(f_{J^*})$.
Using (\ref{eqPPn}),
it can be rewritten
\begin{eqnarray*}
&& R(f_{\hat J}) - R(f_{J^*}) \le(P-P_n) (-2
f_{\hat J} + 2 f_{J^*}).
\end{eqnarray*}
We can control the right-hand side of the last display
using the concentration inequality (\ref{eqbernstein})
with a union bound
over $j=1,\ldots,M$.
For any $ t > 0$, with probability greater than $1-M\exp(-t)$,
\begin{eqnarray*}
(P-P_n) (-2 f_{\hat J} + 2 f_{J^*}) & \le & \max
_{j=1,\ldots,M} (P-P_n) (-2 f_j + 2
f_{J^*})
\\
& \le & \sigma\sqrt{\frac{2t}{n}} + \frac{8 L_0 t}{3 n},
\end{eqnarray*}
where $\sigma^2 = \max_{j=1,\ldots,M} P(-2 f_j + 2 f_{J^*})^2 \le16 L_0^2$.
Setting $x = t - \log M$ yields the desired oracle inequality.
\end{pf*}
By inspecting the short proof above, we see that the slow rate term
$\sqrt{\frac{x+\log M}{n}}$
comes from the variance term in the concentration inequality (\ref
{eqbernstein}).

We can draw two conclusions from Theorems~\ref{thmselectors} and~\ref{thmerm-oi}.
\begin{itemize}
\item In order to achieve faster rates than $\sqrt{\frac{\log M}{n}}$,
we need to look for estimators taking values beyond the discrete set $\{
f_1,\ldots,f_M\}$.
In Section~\ref{sms}, we will consider estimators
taking values in the convex hull of this discrete set.
\item The proof of Theorem~\ref{thmerm-oi} suggests that
a possible way to derive an oracle inequality with fast rates
is to cancel the variance term in the concentration inequality (\ref
{eqbernstein}).
In order to do this, we need some positive gain
on the empirical risk of our estimator.
Namely, for some oracle ${J^*}$ we would like
our estimator $\hat f_n$ to satisfy $R_n(\hat f_n) \le R_n(f_{J^*})$
minus some positive value.
This value is given by the strong convexity of the empirical objective
in Proposition~\ref{propH}.
\end{itemize}

Define the simplex in $\mathbf{R}^M$:
\begin{equation}\label{eqsimplex}
\Lambda^M= \Biggl\{ \theta\in\mathbf{R}^M, \sum
_{j=1}^M\theta_j = 1, \forall j=1,\ldots,
M, \theta_j \ge0 \Biggr\}.
\end{equation}
Given a finite set or \textit{dictionary} $\{f_1,\ldots,f_M\}$,
define for any $\theta\in\Lambda^M$
\begin{equation}\label{eqftheta}
f_\theta= \sum_{j=1}^M
\theta_j f_j.
\end{equation}
In particular, $f_j = f_{e_j}$ where $e_1,\ldots,e_M$
are the vectors of the canonical basis in $\mathbf{R}^M$.

Two classical estimators, the ERM over the convex hull of $f_1,\ldots,f_M$
and the exponential weights aggregate, are known to be sub-optimal
in the regression setting \cite
{dai2012deviation,lecue2009aggregation,lecue2010sharper,lecue2013optimality}.
In the following we show that the same conclusions hold for density
estimation with the $L^2$ risk.

\subsection{ERM over the convex hull}

A first natural estimator valued in the convex hull of the dictionary functions
is the ERM.
However, as in the regression setting \cite{lecue2009aggregation},
this estimator is suboptimal with an unavoidable error term or order
$1/\sqrt{n}$.

\begin{prop}
\label{properm-hull}
Let $\mathcal{X}= \mathbf{R}$ and $\mu$ be the Lebesgue measure on
$\mathbf{R}$.
There exist absolute constants $C_0, C_1, C_2, C_3 > 0$ such
that the following holds.
Let $L>0$.
For any integer $n\ge1$,
there exist a density $f$ bounded by $L$ and a dictionary $\{
f_1,\ldots,f_M\}$ of functions bounded by $2L$,
with $C_0 \sqrt{n} \le M \le C_1 \sqrt{n}$,
such that with probability greater than $1-12 \exp(-C_2 M)$,
\begin{eqnarray*}
&&\llVert f_{{\hat\theta}^{\mathrm{ERM}}} - f \rrVert ^2 \ge\min_{j=1,\ldots,M}
\llVert f_j - f \rrVert ^2 + \frac{C_3 L }{\sqrt{n}},
\end{eqnarray*}
where ${\hat\theta}^{\mathrm{ERM}} := \operatorname{arg}\operatorname
{min}_{\theta\in\Lambda^M}
R_n(f_\theta)$.
\end{prop}

The proof of Proposition~\ref{properm-hull} can be found in Section~\ref{sproofhull}.

\subsection{Exponential weights}

The exponential weights aggregate is known to achieve optimal oracle
inequalities in expectation
when the temperature parameter $\beta> 0$ is chosen carefully \cite
{leung2006information,dalalyan2007aggregation,juditsky2008learning}.
Given prior weights $(\pi_1,\ldots,\pi_M)^T\in\Lambda^M$, it can be
defined as follows:
\begin{eqnarray*}
&& \hat f^{\mathrm{EW}}_\beta= \sum_{j=1}^M
\hat\theta^{\mathrm{EW}, \beta}_j f_j, \qquad  \hat
\theta^{\mathrm{EW}, \beta} \in\Lambda^M, \qquad  \hat\theta^{\mathrm{EW}, \beta}_j
\propto\pi_j \exp \biggl( - \frac
{n}{\beta} R_n(f_j)
\biggr).
\end{eqnarray*}
The following proposition shows that it is suboptimal in deviation for
any temperature,
with a error term of order at least $1/\sqrt{n}$.
This phenomenon was observed in the regression setting \cite
{dai2012deviation,lecue2009aggregation},
and Proposition~\ref{propew} shows that it also holds for density estimation.
As opposed to \cite{dai2012deviation}, the following lower bound
requires only $3$ dictionary functions.

\begin{prop}
\label{propew}
There exist absolute constants $C_0,C_1, N_0 > 0$ such that the
following holds.
Let $\mathcal{X}= \mathbf{R}$ and $\mu$ be the Lebesgue measure on
$\mathbf{R}$.
For all $n\ge N_0, L>0$, there exist a probability density $f$ with
respect to $\mu$,
a dictionary $\{f_1,f_2,f_3\}$ and prior weights $(\pi_1,\pi_2,\pi
_3)\in\Lambda^3$
such that
with probability greater than $C_0$,
\begin{eqnarray*}
&&\bigl\llVert \hat f^{\mathrm{EW}}_\beta- f \bigr\rrVert
^2 \ge\min_{j=1,2,3} \llVert f_j - f
\rrVert ^2 + \frac{C_1 L}{\sqrt{n}}.
\end{eqnarray*}
Furthermore, $\llVert  f \rrVert _\infty\le L$, and $\llVert  f_j \rrVert _\infty\le3L$ for $j=1,2,3$.
\end{prop}

The following proposition shows that the optimality in expectation
cannot hold
if the temperature is below a constant, extending a result from \cite
{lecue2009aggregation}
to the density estimation setting.

\begin{prop}
\label{propew2}
Let $\mathcal{X}= \mathbf{R}$ and $\mu$ be the Lebesgue measure on
$\mathbf{R}$.
There exist absolute constants $c_0, c_1, c_2 > 0$ such
that the following holds.
Let $L>0$.
For any odd integer $n\ge c_0$,
there exist a probability density $f$ with respect to $\mu$ with
$\llVert  f \rrVert _\infty \le L$,
and a dictionary $\{f_1, f_2\}$ with $f_j:\mathcal{X}\rightarrow
\mathbf{R}$
and $\llVert  f_j \rrVert _\infty \le L$ for $j=1,2$
for which the following holds:
\begin{eqnarray*}
&&\mathbb{E}\bigl\llVert \hat f^{\mathrm{EW}}_\beta- f \bigr\rrVert
^2 \ge\min_{j=1,2} \llVert f_j - f
\rrVert ^2 + \frac{c_2 L}{\sqrt{n}} \mbox{ if $\beta\le c_1
L$.}
\end{eqnarray*}
\end{prop}

The proofs of Propositions~\ref{propew} and~\ref{propew2} can be found in
Section~\ref{sproofew}.

\section{Optimal exponential bounds for a penalized procedure}
\label{sms}

\subsection{From strong convexity to a sharp oracle inequality}

In this section, we derive a sharp oracle inequality
for the estimator $f_{\hat\theta}$ where ${\hat\theta}$ is defined
in~(\ref{eqdef-that}).
Define the empirical objective $H_n$ and the estimator ${\hat\theta}$ by
\begin{eqnarray}\label{eqdef-Hn}
H_n(\theta) & =&  \Biggl( \llVert f_\theta \rrVert
^2 - \frac
{2}{n} \sum_{i=1}^n
f_\theta(X_i) \Biggr) + \frac{1}{2}
\operatorname{pen}( \theta ) +\frac{\beta}{n}\sum_{j=1}^M
\theta_j \log\frac{1}{\pi_j},
\\
\label{eqdef-that}
{\hat\theta}& \in& \mathop{\operatorname{argmin}}_{\theta\in
\Lambda^M}
H_n(\theta),
\end{eqnarray}
for some positive constant $\beta$ and
\begin{equation}
\label{eqdef-pen} \forall\theta\in\Lambda^M,\qquad \operatorname{pen}( \theta
) = \sum_{j=1}^M\theta_j
\llVert f_\theta- f_j \rrVert ^2.
\end{equation}
The simplex $\Lambda^M$ and $f_\theta$ are defined in
\eqref{eqsimplex} and \eqref{eqftheta}.

The term
\begin{eqnarray*}
\label{eqK} &&\frac{\beta}{n}\sum_{j=1}^M
\theta_j \log\frac{1}{\pi_j}
\end{eqnarray*}
is a penalty that assigns different weights to the functions $f_j$
according to some prior knowledge given by $\pi_1,\ldots,\pi_M$, in order
to achieve an oracle inequality such as (\ref{eqoracle-ineq}).

The\vspace*{1pt} penalty (\ref{eqdef-pen}) as well as the present procedure are
inspired by recent works
on Q-aggregation in regression models
\cite{rigollet2012kullback,dai2012deviation,lecue2014optimal}.
The choice of the coefficient $\frac{1}{2}$ for the penalty (\ref{eqdef-pen})
is explained in Remark~\ref{remchoice} below.
An intuitive interpretation of the penalty (\ref{eqdef-pen})
can be as follows.
A point $f_\theta$ is in the convex hull of $\{f_1,\ldots,f_M\}$ if
and only if
it is the expectation of a random variable taking values in $\{
f_1,\ldots,f_M\}$.
The penalty (\ref{eqdef-pen}) can be seen as the variance
of such a random variable whose distribution is given by $\theta$.
More precisely, let $\eta$ be a random variable
with $\mathbb{P} ( \eta= j  ) = \theta_j$ for all
$j=1,\ldots,M$.
Denote by $\mathbb{E}_\theta$ the expectation with respect to
the random variable $\eta$.
Then $\mathbb{E}_\theta[ f_\eta]= f_\theta$ and
\begin{eqnarray*}
&&\operatorname{pen}( \theta ) = \mathbb{E}_\theta\bigl\llVert
f_\eta - \mathbb{E}_\theta[ f_\eta] \bigr\rrVert
^2,
\end{eqnarray*}
which is the variance of the random point $f_\eta$.
The penalty (\ref{eqdef-pen}) vanishes at the extreme points:
\begin{eqnarray*}
&&\forall j=1,\ldots, M, \qquad\operatorname{pen}( e_j ) = 0,
\end{eqnarray*}
and $\operatorname{pen}( \theta )$ increases
as $\theta$ moves away from an extreme point $e_j$.
Thus we convexify the optimization problem
over the discrete set $\{f_1,\ldots,f_M\}$
by considering the convex set $\{\mathbb{E}_\theta[f_\eta], \theta
\in
\Lambda^M\}$
which is exactly the convex hull of $\{f_1,\ldots,f_M\}$,
and we penalize by the variance of the random point $f_\eta$.

It is also possible to describe the level sets of the penalty \eqref
{eqdef-pen}.
Assume only in this paragraph that the Gram matrix of $f_1,\ldots,f_M$ is
invertible and let $c\in L^2(\mu)$ be in the linear span of
$f_1,\ldots,f_M$ such that for all $j=1,\ldots,M$, $\int2 c f_j \,d\mu=
\llVert  f_j \rrVert ^2$.
Then simple algebra yields
\begin{eqnarray*}
&&\operatorname{pen}( \theta ) = \llVert c \rrVert ^2 - \llVert
c-f_\theta \rrVert ^2.
\end{eqnarray*}
Thus the level sets of the penalty \eqref{eqdef-pen} are euclidean
balls centered at $c$.

Last, note that $f_{\hat\theta}$ coincides with the $Q$-aggregation procedure
from \cite{dai2012deviation} since
\begin{eqnarray*}
&&\Biggl( \llVert f_\theta \rrVert ^2 - \frac{2}{n}
\sum_{i=1}^n f_\theta(X_i)
\Biggr) + \frac{1}{2} \operatorname{pen}( \theta ) = R_n(
\theta) + \frac{1}{2} \operatorname{pen}( \theta ) = \frac{1}{2}
\Biggl( R_n(\theta) + \sum_{j=1}^M
\theta_j R_n(f_j) \Biggr) .
\end{eqnarray*}

We propose an estimator $f_{\hat\theta}$ based on
penalized empirical risk minimization over the simplex,
with ${\hat\theta}$ defined in (\ref{eqdef-that}).
This estimator satisfies the following oracle inequality.
\begin{thm}
\label{thmms-oi}
Assume that the functions $f_1, \ldots, f_M$ satisfy $\llVert  f_j
\rrVert _\infty\le
L_0$ for all $j=1,\ldots,M$,
and assume that the unknown density $f$ satisfies $\llVert  f \rrVert _\infty\le L$.
Let ${\hat\theta}$ be defined in (\ref{eqdef-that}) with
\begin{eqnarray*}
\beta&=&  4L + \frac{8 L_0}{3}.
\end{eqnarray*}
Then for any $x > 0$, with probability greater than $1-\exp(- x)$,
\begin{eqnarray}\label{eqoracleineqbeta}
R(f_{\hat\theta}) &\le & \min_{j=1,\ldots,M} \biggl(R(f_j)
+ \frac
{\beta
}{n}\log\frac{1}{\pi_j} \biggr) + \frac{\beta x}{n}.
\end{eqnarray}
\end{thm}

The following proposition specifies the property of strong convexity of the
objective function
$H_n(\cdot)$ defined in (\ref{eqdef-Hn}), which is key
to prove Theorem~\ref{thmms-oi}.

\begin{prop}[(Strong convexity of $H_n$)]
\label{propH}
Let $H_n$ and ${\hat\theta}$ be defined by (\ref{eqdef-Hn}) and
(\ref
{eqdef-that}), respectively.
Then for any $\theta\in\Lambda^M$,
\begin{equation}\label{eqthat-H}
H_n({\hat\theta}) \le H_n(\theta) - \tfrac{1}{2}
\llVert f_\theta- f_{\hat\theta} \rrVert ^2.
\end{equation}
\end{prop}
For any $\theta\in\Lambda^M$,
empirical risk minimization only grants the simple inequality
\begin{eqnarray*}
&& R_n({\hat\theta}) \le R_n(\theta),
\end{eqnarray*}
but with Proposition~\ref{propH} we gain the extra term $\frac{1}{2} \llVert  f_\theta- f_{\hat\theta} \rrVert ^2$.
To prove Theorem~\ref{thmms-oi}, we will use this extra term
to compensate the variance term of the concentration
inequality \eqref{eqbernstein2}.
Strong convexity plays an important role in our proofs,
and we believe that our arguments would not work
for loss functions that are not strongly convex such
as the Hellinger distance, the Total Variation distance
or the Kullback--Leibler divergence.

The proof of Proposition~\ref{propH} is given
in Section~\ref{sproofs-H}.
We now give the proof of our main result, which is close
to the proof of Theorem~\ref{thmerm-oi} except that we leverage the strong convexity
of the empirical objective $H_n$.
\begin{pf*}{Proof of Theorem~\protect\ref{thmms-oi}}
Note that $\operatorname{pen}( e_j ) = 0$ for $j=1,\ldots,M$ and let
\begin{eqnarray*}
{J^*}& \in& \mathop{\operatorname{argmin}}_{j=1,\ldots,M} \biggl( \llVert
f_j \rrVert ^2 - 2 \int f_j f \,d\mu+
\frac{\beta}{n} \log\frac{1}{\pi_j} \biggr) = \mathop{\operatorname{argmin}}_{j=1,\ldots,M} \mathbb{E} \bigl[ H_n(e_j)
\bigr].
\end{eqnarray*}
Using (\ref{eqthat-H}) of Proposition~\ref{propH}
\begin{eqnarray*}
H_n({\hat\theta}) - H_n(e_{J^*}) & \le& -
\frac{1}{2} \llVert f_{J^*}- f_{\hat\theta} \rrVert
^2,
\\
R_n({\hat\theta}) + \frac{\beta}{n}\sum
_{j=1}^M{\hat\theta}_j \log
\frac{1}{\pi_j} - R_n(e_{J^*}) - \frac{\beta}{n} \log
\frac{1}{\pi_{J^*}} & \le & - \frac{1}{2} \llVert f_{J^*}-
f_{\hat
\theta} \rrVert ^2 - \frac{1}{2} \operatorname{pen}(
{\hat \theta} )
\\
& =&  - \frac{1}{2} \sum_{j=1}^M{
\hat\theta}_j \llVert f_j - f_{J^*} \rrVert
^2,
\end{eqnarray*}
where we used Proposition~\ref{propbias-variance} with $g=f_{J^*}$ for the
last display.
Using (\ref{eqPPn}), we get
\begin{eqnarray*}
\label{eqlastd}
&& R(f_{\hat\theta}) - R(f_{J^*}) - \frac{\beta}{n}
\log \frac{1}{\pi_{J^*}}  \le Z_n,
\end{eqnarray*}
where
\begin{eqnarray*}
Z_n = (P-P_n) (- 2 f_{\hat\theta}+ 2 f_{J^*})
- \frac{\beta}{n}\sum_{j=1}^M{\hat
\theta}_j \log\frac{1}{\pi_j} - \frac{1}{2} \sum
_{j=1}^M{\hat\theta}_j \llVert
f_j - f_{J^*} \rrVert ^2
\end{eqnarray*}
and the notation $P$ and $P_n$ is defined in (\ref{eqPPn1}) and (\ref
{eqPPn}).
The quantity $Z_n$ is affine in $\theta$
and an affine function over the simplex is maximized at a vertex, so almost
surely,
\begin{eqnarray}
Z_n  &\le &  \max_{\theta\in\Lambda^M}  \Biggl( -2
(P-P_n) (f_\theta-f_{J^*}) - \frac{1}{2}
\sum_{j=1}^M\theta_j \llVert
f_{J^*}- f_j \rrVert ^2 - \frac{\beta}{n}\sum
_{j=1}^M\theta_j \log
\frac{1}{\pi_j} \Biggr)
\nonumber
\\[-8pt]
\label{eqmax-on-k}
\\[-8pt]
\nonumber
&=& \max_{k=1,\ldots,M}  \biggl( -2 (P-P_n)
(f_k - f_{J^*}) - \frac
{1}{2} \llVert
f_k -f_{J^*} \rrVert ^2 - \frac{\beta}{n}
\log\frac{1}{\pi
_k} \biggr).
\end{eqnarray}
Let $k=1,\ldots,M$ fixed.
Applying Proposition~\ref{propstochastic-term} with $g=-2(f_k - f_{J^*})$ and
$\pi=\pi_k$ yields
\begin{eqnarray*}
&& \mathbb{P} \biggl( -2 (P-P_n) (f_k -f_{J^*}) -
\frac{1}{2} \llVert f_k - f_{J^*} \rrVert
^2 - \frac{\beta}{n} \log\frac{1}{\pi
_k} > \frac{\beta x}{n}
\biggr) \le\pi_k \exp(-x).
\end{eqnarray*}
To complete the proof,
we use a union bound on $k=1,\ldots,M$ together with $\sum_{j=1}^M\pi
_j = 1$
and~(\ref{eqmax-on-k}):
\begin{eqnarray*}
&&\mathbb{P} \biggl( Z_n > \frac{\beta x}{n} \biggr) \le\sum
_{k=1}^M \pi_k \exp(-x) = \exp(-x).
\end{eqnarray*}
\upqed\end{pf*}

\begin{remark}[(Choice of the coefficient of the penalty (\protect\ref{eqdef-pen}))]
\label{remchoice}
Let $\nu\in(0,1)$.
With minor modifications to the proof of Theorem~\ref{thmms-oi},
it can be shown that
the oracle inequality (\ref{eqoracleineqbeta})
still holds with
\begin{eqnarray*}
\beta& =&  \frac{2L}{\min(\nu,1-\nu)} + \frac{8 L_0}{3},
\\
H_n(\theta) & =& \Biggl( \llVert f_\theta \rrVert
^2 - \frac
{2}{n} \sum_{i=1}^n
f_\theta(X_i) \Biggr) + \nu\operatorname{pen}( \theta ) +
\frac{\beta}{n}\sum_{j=1}^M
\theta_j \log\frac{1}{\pi_j},
\\
{\hat\theta}& \in & \mathop{\operatorname{argmin}}_{\theta\in
\Lambda^M}
H_n(\theta).
\end{eqnarray*}
The oracle inequality (\ref{eqoracleineqbeta}) is best when $\beta$
is small.
Thus, the choice $\nu=\frac{1}{2}$ is natural since it
minimizes the value of $\beta$.
\end{remark}

The optimization problem \eqref{eqdef-that}
is a quadratic program, for which efficient algorithms exist.
We refer to \cite{dai2012deviation}, Section~4,  for an analysis
of the statistical performance of an algorithm
that approximately solves a optimization problem similar to
\eqref{eqdef-that} in the regression setting.

The estimator ${\hat\theta}$ of Theorem~\ref{thmms-oi} is not adaptive
since its construction relies on $L$,
an upper bound of the sup-norm of the unknown density.
However, in the case of the uniform prior
$\pi_j = 1/M$ for all $j=1,\ldots,M$,
Corollary~\ref{coruniform} below provides an estimator
which is fully adaptive:
its construction depends only on the functions $f_1, \ldots, f_M$
and the data $X_1, \ldots, X_n$.
A similar adaptivity property was
observed in \cite{lecue2014optimal} in the regression setting.

\begin{cor}[(Adaptive estimator)]
\label{coruniform}
Assume that the functions $f_1, \ldots, f_M$ satisfy $\llVert  f_j
\rrVert _\infty\le
L_0$ for all $j=1,\ldots,M$,
and assume that the unknown density $f$ satisfies $\llVert  f \rrVert _\infty\le L$.
Let
\begin{equation}\label{eqdef-that-uniform}
{\hat\theta}\in\mathop{\operatorname{argmin}}_{\theta\in
\Lambda^M} \Biggl(
\llVert f_\theta \rrVert ^2 - \frac{2}{n} \sum
_{i=1}^n f_\theta(X_i) \Biggr)
+ \frac{1}{2} \operatorname{pen}( \theta ).
\end{equation}
Then for any $x > 0$, with probability greater than $1-\exp(- x)$,
\begin{eqnarray*}
&& R(f_{\hat\theta}) \le\min_{j=1,\ldots,M} R(f_j) +
\biggl( 4L + \frac{8 L_0}{3} \biggr) \frac{\log(M) + x}{n}.
\end{eqnarray*}
\end{cor}
\begin{pf}
With the uniform prior, $\pi_j = 1/M$ for all $j=1,\ldots, M$, the quantity
\begin{eqnarray*}
&&\frac{\beta}{n}\sum_{j=1}^M
\theta_j \log\frac{1}{\pi_j} = \frac
{\beta}{n} \log M
\end{eqnarray*}
is independent of $\theta\in\Lambda^M$.
The minimizer (\ref{eqdef-that-uniform}) is also a minimizer
of the empirical objective~(\ref{eqdef-Hn}) used in Theorem~\ref{thmms-oi}.
Thus, the estimator $f_{\hat\theta}$ satisfies (\ref{eqoracleineqbeta})
which completes the  proof.
\end{pf}

Corollary~\ref{coruniform} is in contrast to
methods related to exponential weights
such as the mirror averaging algorithm from
\cite{juditsky2008learning} as these methods rely
on the knowledge of the sup-norm of the unknown density.
The method presented here is an improvement in two aspects.
First,
the estimator of Corollary~\ref{coruniform} is fully data-driven.
Second, the sharp oracle inequality is satisfied
not only in expectation, but also in deviation.

However, the method of Theorem~\ref{thmms-oi}
loses this adaptivity property
when a non-uniform prior is used,
and we do not know if it is possible to build an optimal
and fully adaptive estimator for non-uniform priors.

\subsection{A lower bound with exponential tails}

The following lower bound
shows that the sharp oracle inequality of Corollary~\ref{coruniform}
cannot be improved both in the rate and in the tail of the deviation.
\begin{thm}[(Lower bounds with optimal deviation term)]
\label{thmlower-exp}
Let $M\ge2, n\ge1$ be two integers and
let a real number $x\ge0$ satisfy
\begin{eqnarray*}
&&\frac{\log(M) + x}{n} < 3.
\end{eqnarray*}
Let $L>0$ and $d\ge1$. Let $\mathcal{F}$ be the class of densities
$f$ with respect to
the Lebesgue measure on $\mathbf{R}^d$ such that $\Vert f\Vert_\infty
\le L$.

Then there exist $M$ functions $f_1,\ldots,f_M$
in $L^2(\mathbf{R}^d)$ with
$\llVert  f_j \rrVert _\infty \le L$ satisfying
\begin{eqnarray*}
&&\inf_{\hat T_n} \sup_{f \in\mathcal{F}} \mathbb{P}_f
\biggl( \llVert \hat T_n -f \rrVert ^2 - \min
_{j=1,\ldots,M} \llVert f_j -f \rrVert ^2 >
\frac
{L}{24} \biggl( \frac{\log(M) + x}{n} \biggr) \biggr) \ge\frac{1}{24}
\exp(-x),
\end{eqnarray*}
where the infimum is taken over all estimators $\hat T_n$ and
$\mathbb{P}_f$ denotes the probability with respect to $n$ i.i.d.
observations with density $f$.
\end{thm}
Notice that the restriction
$\frac{\log(M) + x}{n} < 3$
is natural since the estimator $\hat T_n^* \equiv0$ achieves a
constant error term
and is optimal in the region
$\frac{\log(M) + x}{n} > c$ for some absolute constant $c$.
Indeed, as the unknown density satisfies $\llVert  f \rrVert
_\infty\le L$, we have
with probability $1$:
\begin{eqnarray}
\bigl\llVert \hat T_n^* - f \bigr\rrVert ^2 &=&  \llVert f
\rrVert ^2 \le L  \le\inf_{j=1,\ldots,M} \llVert
f-f_j \rrVert ^2 + L,
\nonumber
\\[-8pt]
\label{eqdummy0}
\\[-8pt]
\nonumber
R\bigl(\hat T_n^*\bigr) & \le & \inf_{j=1,\ldots,M}
R(f_j) + L.
\end{eqnarray}
Thus it is impossible to get the lower bound of Theorem~\ref{thmlower-exp}
for arbitrarily large $\frac{x+\log M}{n}$.

\subsection{Weighted loss and unboundedness}

The previous strategy based on penalized risk minimization over the simplex
can be applied to handle unbounded densities or unbounded dictionary functions,
if we use a weighted loss.

Let $w:\mathcal{X}\rightarrow\mathbf{R}^+$ be a measurable function with
respect to $\mu$.
Define the norm (or semi-norm if $w$ is zero on a set of positive measure)
\begin{eqnarray*}
&& \llVert g \rrVert ^2_w = \int
g^2 w \,d\mu,  \qquad  \forall g\in L^2(\mu).
\end{eqnarray*}
Then we can define the estimator $f_{\hat\theta}$ where
\begin{eqnarray*}
&&{\hat\theta}= \mathop{\operatorname{argmin}}_{\theta\in
\Lambda^M}
V_n(\theta), \qquad   V_n(\theta) = P_n
\Biggl( \llVert f_\theta \rrVert ^2_w -
\frac{2}{n} \sum_{i=1}^n
f_\theta(X_i)w(X_i) + \frac{1}{2} \sum
_{j=1}^M \theta_j \llVert
f_j-f_\theta \rrVert ^2_{w} \Biggr).
\end{eqnarray*}
The function $V_n$ is strongly convex with respect to the new norm
$\llVert \cdot \rrVert ^2_w$.
As in the proof of Theorem~\ref{thmms-oi}, this leads to
\begin{eqnarray*}
\llVert f_{\hat\theta}- f \rrVert ^2_{w} &\le & \llVert
f_{J^*}- f \rrVert ^2_{w} + \max
_{k=1,\ldots,M} \delta_k,
\\
 \delta_k
& := & (P-P_n) \bigl(-2(f_{J^*}-f_k)w\bigr) -
\frac{1}{2} \llVert f_{J^*}- f_k \rrVert
^2_{w}.
\end{eqnarray*}
If for some $L, L_0>0$, $\llVert  wf \rrVert _\infty\le L$ and
$\max_{j=1,\ldots,M}
\llVert  w f_j \rrVert _\infty\le L_0$, then
\begin{eqnarray*}
&&\delta_k \le-2 (P-P_n) \bigl((f_k -
f_{J^*})w\bigr) - \frac{1}{2L} \mathbb{E} \bigl[
\bigl(f_k(X) - f_{J^*}(X)\bigr)^2
w(X)^2\bigr].
\end{eqnarray*}
We apply \eqref{eqbernstein2} to the random variables $(f_k-f_{J^*}
)(X_i)w(X_i)$,
which are almost surely bounded by $L_0$.
Using the union bound on $k=1,\ldots,M$, we obtain
$\max_{k=1,\ldots,M} \delta_k \le\beta(x+\log M)/n$
with probability greater than $1-\exp(-x)$.
and thus
\begin{eqnarray*}
&&\llVert f_{\hat\theta}- f \rrVert ^2_{w} \le\llVert
f_{J^*}- f \rrVert ^2_{w} + \beta \biggl(
\frac{x + \log M}{n} \biggr),
\end{eqnarray*}
where $\beta= c  (L+L_0)$ for some numerical constant $c>0$.

\subsection{Differences and similarities with regression
problems}
\label{sdiff-reg}

Here, we discuss differences and similarities between
aggregation of
density and regression estimators.
Some notation is needed in order to compare these settings.

We first define some notation related to the Density Estimation (DE) framework
studied in the present paper.
Let $X$ be a random variable with density $f$ absolutely continuous
with respect to the measure $\mu$,
let $\mathcal{D}^\mathrm{DE} = \{f_1,\ldots,f_M\}$ be a subset of
$L^2(\mu)$ and define for all $g\in L^2(\mu)$ and $x\in\mathcal{X}$,
\begin{eqnarray*}
\Vert g \Vert^2 &=& \int g^2 \,d\mu,\qquad
l_g^{\mathrm{DE}}(x) = \llVert g \rrVert ^2
- 2g(x),
\\
g^* &=& \mathop{\operatorname{argmin}}_{g\in\mathcal{D}^{\mathrm{DE}}} \Vert g -
f\Vert^2 = \mathop{\operatorname{argmin}}_{g\in\mathcal{D}^{\mathrm{DE}}}
\mathbb{E}\bigl[l_g^{\mathrm{DE}}(X)\bigr].
\end{eqnarray*}
Given $n$ i.i.d. observations $X_1,\ldots,X_n$ and some fixed function $g$,
one can use the empirical risk $P_n ( l_g^{\mathrm{DE}} ) = \sum_{i=1}^n ({1}/{n})l_g^{\mathrm{DE}}(X_i)$.

We now define similar notation for the regression problem with the
$L^2$ loss.
Let $(X,Y)$ be a random couple valued in $\mathcal{X}\times\mathbf{R}$,
let $P_X$ be the probability measure of $X$,
let $f$ be the true regression function defined by $f(x) = \mathbb{E}[Y|X=x]$,
let $\mathcal{D}^\mathrm{R} = \{f_1,\ldots,f_M\}$ be a subset of $L^2(P_X)$
and define for all $g\in L^2(P_X)$,
\begin{eqnarray*}
&&\Vert g \Vert^2_{P_X} = \mathbb{E}\bigl[g(X)^2
\bigr],\qquad  g^* = \mathop{\operatorname{argmin}}_{g\in\mathcal{D}^{\mathrm{R}}} \Vert g -
f\Vert^2_{P_X}.
\end{eqnarray*}
For Regression with Unknown Design (RUD) that is, when the distribution
of the design $X$ is not known to the statistician,
a natural choice for the loss function $l_g$ is
\begin{eqnarray*}
&& l_g^{\mathrm{RUD}}(x,y) = \bigl(g(x) -y\bigr)^2,\qquad
\forall x,y\in\mathcal{X}\times\mathbf{R},
\end{eqnarray*}
and the oracle $g^*$ defined above satisfies $g^* = \operatorname
{arg}\operatorname{min}_{g\in
\mathcal{D}^{\mathrm{R}}} \mathbb{E}[l_g^{\mathrm{RUD}}(X,Y)]$.
For Regression with Known Design (RKD),
the quantity $\llVert  g \rrVert ^2_{P_X}$ is accessible for
all $g$.
Thus, we can define the loss
\begin{eqnarray*}
&& l_g^{\mathrm{RKD}}(x,y) = \llVert g \rrVert ^2_{P_X}
- 2 g(x)y, \qquad \forall x,y\in\mathcal{X}\times\mathbf{R},
\end{eqnarray*}
and the oracle $g^*$ satisfies $g^* = \operatorname{argmin}_{g\in\mathcal
{D}^{\mathrm{R}}} \mathbb{E}[l_g^{\mathrm{RKD}}(X,Y)]$.
Thus, two natural functions $l_g$ arise in the regression context,
depending on whether the distribution of the design is known or unknown.
Given $n$ i.i.d. observations $(X_i, Y_i)$ with the same distribution
as $(X,Y)$,
the empirical quantities $P_n ( l_g^{\mathrm{RUD}} )$
and $P_n ( l_g^{\mathrm{RKD}} )$ can be used to infer the true
regression function $f$.
An estimator constructed using the quantity $P_n ( l_g^{\mathrm{RKD}}
)$ is used,
for example, in
\cite{tsybakov2003optimal} for the problem of linear and convex aggregation.

\subsubsection*{Linear or quadratic empirical process}
The empirical process $(P_n-P)(l_g-l_{g^*})$ indexed by $g$ plays an
important role
in the proofs of Theorems~\ref{thmerm-oi} and \ref{thmms-oi}.
This empirical process also appears in the analysis \cite
{lecue2014optimal} for regression with unknown design with the loss
$l_g^{\mathrm{RUD}}$.
For density estimation and regression with known design, this empirical
process is linear in $g$:
\begin{eqnarray*}
(P_n-P) \bigl(l_g^\mathrm{DE}-l_{g^*}^\mathrm{DE}
\bigr) &=& - 2 (P_n-P) \bigl(g-g^*\bigr),
\\
 (P_n-P)
\bigl(l_g^\mathrm{RKD}-l_{g^*}^\mathrm{RKD}
\bigr) &=& - 2 (P_n-P)\bigl[\bigl(g-g^*\bigr)\dot{y}\bigr],
\end{eqnarray*}
where the function $\dot{y}(\cdot)$ above is defined by
$\forall x,y\in\mathcal{X}\times\mathbf{R},\dot{y}(x,y) = y$.
For regression when the design is unknown,
the empirical process is quadratic in the class member $g$.
To control the behavior of this quadratic empirical process,
the contraction principle is used in \cite{lecue2014optimal},
whereas this principle is not needed for density estimation or
regression when the distribution of the design is known.


\subsubsection*{The penalty \protect\eqref{eqdef-pen} and its coefficient}
In the regression problem when the distribution is known,
given a dictionary of potential regression functions $\{f_1,\ldots,f_M\}$,
the quantity
\begin{equation}
\label{eqpen-rkd} \sum_{j=1}^M
\theta_j \llVert f_j-f_\theta \rrVert
^2_{P_X},
\end{equation}
is accessible
and a procedure similar to the one proposed in Theorem~\ref{thmms-oi}
and Corollary~\ref{coruniform} can be constructed,
with the penalty coefficient $1/2$
which is a natural choice as explained in Remark~\ref{remchoice}.
For regression with unknown design, the above penalty cannot be computed:
the procedure \cite{lecue2014optimal} for the $L^2$ loss is the
estimator $f_{\hat\theta}$ where
\begin{eqnarray*}
{\hat\theta}& = & \mathop{\operatorname{argmin}}_{\theta\in
\Lambda^M} \bigl(
P_n \bigl( l_{f_\theta}^\mathrm{RUD} \bigr) + \nu
P_n (f_j - f_\theta)^2 \bigr)
\\
&= & \mathop{\operatorname{argmin}}_{\theta\in\Lambda^M} \Biggl(
\frac{1}{n}\sum_{i=1}^n
\bigl(Y_i - f_\theta(X_i)\bigr)^2 +
\frac{\nu}{n} \sum_{i=1}^n
(f_j - f_\theta)^2(X_i) \Biggr),
\end{eqnarray*}
for some coefficient $\nu\in(0,1)$ and where we chose the uniform
prior for clarity.
Thus, the procedure \cite{lecue2014optimal} can be formulated as a
penalized procedure
where the penalty is the empirical counterpart of \eqref{eqpen-rkd}
with the coefficient $\nu$.
Although $1/2$ is a natural choice for
regression with known design and density estimation,
for regression with unknown design
the expression of the optimal coefficient is more intricate
\cite{lecue2014optimal}, minimize $\beta$ in (1.4).

\subsubsection*{Sketch of proof for the regression model with known design}
In order to show the similarities
between density estimation and regression
problems when the design is known,
we now give the main ideas to derive an oracle inequality
similar to Corollary~\ref{coruniform} for regression with known design.
Note that the framework studied in \cite{lecue2014optimal} does not
cover the
estimator defined below, since the function $l_g^\mathrm{RKD}$ depends
on the quantity $\llVert  g \rrVert ^2_{P_X}$.
Given $n$ i.i.d. observations $(X_1,Y_1),\ldots,(X_n,Y_n)$, define
\begin{eqnarray*}
&&{\hat\theta}= \mathop{\operatorname{argmin}}_{\theta\in
\Lambda^M}
V_n(\theta),   \qquad V_n(\theta) = P_n
\bigl( l_{f_\theta}^\mathrm{RKD} \bigr) + \frac
{1}{2} \sum
_{j=1}^M \theta_j \llVert
f_j-f_\theta \rrVert ^2_{P_X}.
\end{eqnarray*}
Analogously to the argument of Proposition~\ref{propH}, we note that
the function $V_n$ is strongly convex and
$V_n({\hat\theta}) \le V_n(e_{J^*}) - \frac{1}{2} \llVert
f_{J^*}- f_{\hat\theta} \rrVert ^2_{P_X}$ for any ${J^*}=1,\ldots,M$.
As in the proof of Theorem~\ref{thmms-oi}, this leads to
\begin{eqnarray*}
\llVert f_{\hat\theta}- f \rrVert ^2_{P_X}  &\le & \llVert
f_{J^*}- f \rrVert ^2_{P_X} + \max
_{k=1,\ldots,M} \delta_k, \\
\delta_k
&:=& (P-P_n) \bigl(l_{f_k}^\mathrm{RKD} -
l_{f_{J^*}
}^\mathrm{RKD}\bigr) - \frac{1}{2} \llVert
f_{J^*}- f_k \rrVert ^2_{P_X}.
\end{eqnarray*}
As explained above,
when the distribution of the design is known,
the empirical process is linear in $f_k-f_{J^*}$:
\begin{eqnarray*}
&&\delta_k = -2 (P-P_n) \bigl((f_k -
f_{J^*})\dot{y}\bigr) - \tfrac{1}{2} \llVert f_k-f_{J^*}
\rrVert ^2_{P_X}.
\end{eqnarray*}
If for some $b>0$, $|Y|\le b$ and $\max_{j=1,\ldots,M} |f_j(X)| \le b$
almost surely, then
\begin{eqnarray*}
&&\delta_k \le-2 (P-P_n) \bigl((f_k -
f_{J^*})\dot{y}\bigr) - \frac{1}{2b^2} \mathbb{E}
\bigl[Y^2\bigl(f_k(X) - f_{J^*}(X)\bigr)\bigr].
\end{eqnarray*}
Using \eqref{eqbernstein2} and the union bound on $k=1,\ldots,M$, we obtain
$\max_{k=1,\ldots,M} \delta_k \le\beta(x+\log M)/n$
with probability greater than $1-\exp(-x)$
and thus
\begin{eqnarray*}
&&\llVert f_{\hat\theta}- f \rrVert ^2_{P_X} \le\llVert
f_{J^*}- f \rrVert ^2_{P_X} + \beta \biggl(
\frac{x + \log M}{n} \biggr),
\end{eqnarray*}
where $\beta= c  b^2$ for some numerical constant $c>0$.

In conclusion,
the density estimation framework studied in the present paper is close
to the regression problem when the distribution of the design is known,
while it presents several differences with the regression problem when
the design is not known.

\section{Minimax optimality in deviation}
\label{sminimax}

The goal of this section is to state a minimax optimality result
based on the lower bound of Theorem~\ref{thmlower-exp} and
the sharp oracle inequality of Corollary~\ref{coruniform}.
In this section, the underlying measure $\mu$ is the Lebesgue measure
on $\mathbb{R}^d$ for some integer $d\ge1$.

Minimax optimality in model selection type aggregation is usually
defined in expectation \cite{tsybakov2003optimal}, by studying the quantity
\begin{eqnarray*}
&&\mathop{\sup_{f_j\in\mathcal{F}}}_{j=1,\ldots, M} \inf_{\hat
T_n}
\sup_{f\in\mathcal{F}_d} \Bigl( \mathbb{E}R(\hat T_n) - \inf
_{j=1,\ldots,M} R(f_j) \Bigr),
\end{eqnarray*}
where the infimum is taken over all estimators $\hat T_n$,
$\mathcal{F}$ is a class of possible functions for the dictionary
and $\mathcal{F}_d$ is the class of all densities satisfying some
general constraints.

Let $\mu$ be the Lebesgue measure on $\mathbf{R}^d$
and for some $L>0$,
let $\mathcal{F} = \{g\in L^2(\mu), \llVert  g \rrVert
_\infty \le L\}$
and $\mathcal{F}_d$ be the set of all densities $f$ with respect to
$\mu$ satisfying $\llVert  f \rrVert _\infty\le L$.
Then, by an integration argument, Corollary~\ref{coruniform} and Theorem~\ref
{thmlower-exp} provide the following bounds
for some absolute constant $c,C>0$ and any $M\ge2,n\ge1$:
\begin{eqnarray*}
&& c \frac{L \log M}{n} \le\mathop{\sup_{f_j\in\mathcal{F}}}_{j=1,\ldots, M}
\inf_{\hat T_n} \sup_{f\in\mathcal{F}_d} \Bigl( \mathbb{E} R(\hat
T_n) - \inf_{j=1,\ldots,M} R(f_j) \Bigr) \le C
\frac{L
\log M}{n}.
\end{eqnarray*}
This shows that $\frac{L\log M}{n}$ is the optimal rate of convergence
in expectation for model selection type aggregation
under the boundedness assumption.

But our results are stronger that the above optimality in expectation since
the deviation term in the sharp oracle inequality
of Corollary~\ref{coruniform}
and in the lower bound of Theorem~\ref{thmlower-exp} are the same up to a
numerical constant.

The central quantity when dealing with optimality in deviation is, for $t>0$,
\begin{eqnarray*}
&&\mathop{\sup_{f_j\in\mathcal{F}}}_{j=1,\ldots, M} \inf_{\hat T_n}
\sup_{f\in\mathcal{F}_d} \mathbb{P} \Bigl( R(\hat T_n) - \inf
_{j=1,\ldots,M} R(f_j) > t \Bigr).
\end{eqnarray*}
The results of Section~\ref{sms} provide upper and lower bounds for this quantity.

We propose the following definition of minimax optimality in deviation.
\begin{mydef}[(Minimax optimality in deviation)]
\label{defminimax-optimality}
Let $\mathcal{F}$ be a subset of $L^2(\mu)$ and $\mathcal{F}_d$
be a set of densities with respect to the measure $\mu$.
Let $\mathcal{E}_n$ be a set of estimators.
Denote by $\mathbf{P}^{n,M}_{\mathcal{E}_n, \mathcal{F}, \mathcal
{F}_d}(t)$ the quantity
\begin{eqnarray*}
&& \mathbf{P}^{n,M}_{\mathcal{E}_n, \mathcal{F}, \mathcal{F}_d}(t) = \mathop{\sup
_{f_j\in\mathcal{F}}}_{j=1,\ldots, M} \inf_{\hat
T_n \in\mathcal{E}_n} \sup
_{f\in\mathcal{F}_d} \mathbb {P} \Bigl( R(\hat T_n) - \inf
_{j=1,\ldots,M} R(f_j) > t \Bigr).
\end{eqnarray*}
A function $p_{n,M}(\cdot)$ is called optimal tail distribution over
$(\mathcal{E}_n, \mathcal{F}, \mathcal{F}_d)$
if for any $n\ge1,M\ge2$ and any $t>0$,
\begin{eqnarray*}
&& c p_{n,M}\bigl(c' t\bigr) \le \mathbf{P}^{n,M}_{\mathcal{E}_n, \mathcal{F}, \mathcal{F}_d}(t)
\le p_{n,M}(t),
\end{eqnarray*}
where $c,c'>0$ are constants independent of $n,M$ and $t$.
\end{mydef}

The following proposition is a direct consequence of Corollary~\ref
{coruniform} and Theorem~\ref{thmlower-exp}.

\begin{prop}
\label{propoptim}
Let $M \ge2, n\ge1$ and $L>0$.
Let $\mathcal{F}=\{g\in L^2(\mathbf{R}^d), \llVert  g \rrVert
_\infty\le L\}$
and $\mathcal{F}_d$ be the set of all densities $f$ with respect to
the Lebesgue measure on $\mathbf{R}^d$ with $\llVert  f \rrVert _\infty\le L$.
Let $\mathcal{E}_n$ be the set of all estimators.
Define
\begin{eqnarray*}
&& p_{n,M}(t) = M \exp \biggl(- \frac{3 t n}{20 L} \biggr)
\mathbf{1}_{[0,L]}(t),
\end{eqnarray*}
where $\mathbf{1}_A$ denotes the indicator function of the set $A$.
Then for all $t>0$,
\begin{eqnarray*}
&& \tfrac{1}{24} p_{n,M} ( 160 t ) \le \mathbf{P}^{n,M}_{\mathcal{E}_n, \mathcal{F}, \mathcal{F}_d}(t)
\le p_{n,M} ( t ).
\end{eqnarray*}
Thus, $p_{n,M}(\cdot)$
is an optimal tail distribution over $(\mathcal{E}_n, \mathcal{F},
\mathcal{F}_d)$
according to  Definition~\ref{defminimax-optimality}.
\end{prop}

\begin{pf}
The regime $t>L$ corresponds to the trivial case where
(\ref{eqdummy0}) holds and
$\hat T_n^* = 0$ is an optimal estimator.
In this regime\vspace*{1pt} $p_{n,M}(t) = 0$.

For $t\le L$,
by setting $t=\beta\frac{\log(M) + x}{n} = \frac{20L}{3} \frac
{\log(M) + x}{n}$ in Corollary~\ref{coruniform}, we get
\begin{eqnarray*}
&& \mathbf{P}^{n,M}_{\mathcal{E}_n, \mathcal{F}, \mathcal{F}_d} \le p_{n,M} ( t )
\end{eqnarray*}
while Theorem~\ref{thmlower-exp} implies that
\begin{eqnarray*}
&&\frac{1}{24} p_{n,M} \biggl( \frac{24\cdot20}{3} t \biggr) \le
\mathbf{P}^{n,M}_{\mathcal{E}_n, \mathcal{F}, \mathcal{F}_d}(t).
\end{eqnarray*}
\upqed\end{pf}

Similarly, the results of Section~\ref{sselectors} imply the following proposition.
\begin{prop}
Let $M \ge2, n\ge1$ and $L>0$.
Let $\mathcal{F}=\{g\in L^2(\mathbf{R}^d), \llVert  g \rrVert
_\infty\le L\}$
and $\mathcal{F}_d$ be the set of all densities $f$ with respect to
the Lebesgue measure on $\mathbf{R}^d$ with $\llVert  f \rrVert _\infty\le L$.
Let
$\mathcal{S}_n$ be the set of all selectors, that is, the measurable functions
valued in the discrete set $\{f_1,\ldots,f_M\}$.
Define
\begin{eqnarray*}
&& q_{n,M}(t) = M \exp \biggl(-\frac{t^2 n}{L^2 (4\sqrt{2} + 8/3)^2} \biggr)
\mathbf{1}_{[0,L]}(t),
\end{eqnarray*}
where $\mathbf{1}_A$ denotes the indicator function of the set $A$.
Then for all $t>0$,
\begin{eqnarray*}
&& \tfrac{1}{24} q_{n,M} \bigl( \sqrt{3} (4\sqrt{2} + 8/3) t \bigr)
\le \mathbf{P}^{n,M}_{\mathcal{S}_n, \mathcal{F}, \mathcal{F}_d}(t) \le q_{n,M} ( t ).
\end{eqnarray*}
Thus, $q_{n,M}(\cdot)$
is an optimal tail distribution over $(\mathcal{S}_n, \mathcal{F},
\mathcal{F}_d)$
according to Definition~\ref{defminimax-optimality}.
\end{prop}

\begin{pf}
The regime $t>L$ can be treated similarly as in the proof of Proposition~\ref{propoptim}.

For $t\le L$, let $t=L (4\sqrt{2} + 8/3) \sqrt{\frac{x+\log M}{n}}$
in Theorem~\ref{thmerm-oi}.
For this definition of $t$ and $x$, $1 \ge\sqrt{\frac{x+\log
M}{n}}\ge\frac{x+\log M}{n}$.
Then
\begin{eqnarray*}
&&\mathbf{P}^{n,M}_{\mathcal{S}_n, \mathcal{F}, \mathcal{F}_d}(t) \le q_{n,M} ( t )
\end{eqnarray*}
and Theorem~\ref{thmselectors} implies
\begin{eqnarray*}
&& \tfrac{1}{24} q_{n,M} \bigl( \sqrt{3} (4\sqrt{2} + 8/3) t \bigr)
\le \mathbf{P}^{n,M}_{\mathcal{S}_n, \mathcal{F}, \mathcal{F}_d}(t).
\end{eqnarray*}
\upqed\end{pf}

\section{Proofs}
\label{sproofs}

\subsection{Bias-variance decomposition}

As discussed in Section~\ref{sms}, the penalty can be viewed as the variance
of a random element of the discrete set $\{f_1,\ldots,f_M\}$
and it satisfies the following bias-variance decomposition.
\begin{prop}
\label{propbias-variance}
For any $g\in L^2(\mu)$ and $\theta\in\Lambda^M$,
\begin{equation}\label{eqbias-variance}
\sum_{j=1}^M\theta_j \llVert
f_j - g \rrVert ^2 = \llVert f_\theta- g
\rrVert ^2 + \operatorname{pen}( \theta ),
\end{equation}
where $\operatorname{pen}( \cdot )$ is the penalty defined in (\ref
{eqdef-pen}).
\end{prop}
\begin{pf}
Let $\eta$ be a random variable
with $\mathbb{P} ( \eta= j  ) = \theta_j$ for all
$j=1,\ldots,M$.
Denote by $\mathbb{E}_\theta$ the expectation with respect to
the random variable $\eta$.
Then $\mathbb{E}_\theta[ f_\eta]= f_\theta$ and
the bias-variance decomposition yields
\begin{eqnarray*}
&&\mathbb{E}_\theta\llVert f_\theta- g \rrVert ^2 =
\bigl\llVert g-\mathbb{E}_\theta[ f_\eta] \bigr\rrVert
^2 + \mathbb{E}_\theta\bigl\llVert f_\eta-
\mathbb{E}_\theta[ f_\eta] \bigr\rrVert ^2,
\end{eqnarray*}
which is exactly the desired result.
\end{pf}

\subsection{Concentration inequalities}

\begin{prop}
\label{propbernstein}
Let $Y_1, \ldots, Y_n$ be independent random variables,
such that almost surely, for all $i$, $|Y_i - \mathbb{E}Y_i| \le b$.
Then for all $x > 0$,
\begin{eqnarray} \label{eqbernstein}
&& \mathbb{P} \Biggl( \sum_{i=1}^nY_i
- \mathbb{E}Y_i > \sqrt{ 2 x v } + \frac{bx}{3} \Biggr)  \le
\exp(-x),
\end{eqnarray}
where $v =\sum_{i=1}^n \mathbb{V} ( Y_i  )$.
\end{prop}
Proposition~\ref{propbernstein} is close to Bennett and Bernstein inequalities.
A proof can be found in \cite{massart2007concentration}, Section~2.2.3, (2.20) with
$c=b/3$.

The following one-sided concentration inequality is a direct consequence
of Proposition~\ref{propbernstein}
and the inequality $2\sqrt{uv} \le\frac{u}{a} + av$ for all $a,u,v>0$.
Under the same assumptions as Proposition~\ref{propbernstein} above,
for all $x > 0$ and any $a>0$,
\begin{eqnarray}\label{eqbernstein2}
&& \mathbb{P} \Biggl( \frac{1}{n}\sum_{i=1}^nY_i
- \mathbb{E}Y_i - a \mathbb{V} ( Y_i ) > \biggl(
\frac{1}{2a} + \frac{b}{3} \biggr)\frac{x}{n} \Biggr)  \le
\exp(-x).
\end{eqnarray}

\begin{prop}
\label{propstochastic-term}
Let $X_1,\ldots,X_n$ be i.i.d. observations drawn from the density $f$
with $\llVert  f \rrVert _\infty\le L$.
Let $g\in L^2(\mu)$ with $\llVert  g \rrVert _\infty\le4L_0$.
Let $\beta= 4L + \frac{8 L_0}{3}$.
Define
\begin{eqnarray*}
&& \zeta_n = (P-P_n)g - \frac{1}{8} \llVert g
\rrVert ^2 - \frac{\beta}{n} \log \frac{1}{\pi},
\end{eqnarray*}
where the notation $P$ and $P_n$ is defined in (\ref{eqPPn1}).
Then for all $x > 0$,
\begin{eqnarray*}
&&\mathbb{P} \biggl( \zeta_n > \frac{\beta x}{n} \biggr) \le\pi
\exp(-x).
\end{eqnarray*}
\end{prop}

\begin{pf}
As the unknown density $f$ is bounded by $L$,
\begin{eqnarray*}
\mathbb{V} \bigl( g(X_1) \bigr) & \le&  P\bigl(g^2\bigr)
= \int g^2 f \,d\mu\le L \llVert g \rrVert ^2,
\\
-\frac{1}{8} \llVert g \rrVert ^2 & \le& - \frac{1}{8L}
\mathbb{V} \bigl( g(X_1) \bigr).
\end{eqnarray*}
Thus, almost surely
\begin{eqnarray*}
&&\zeta_n \le (P-P_n)g -\frac{1}{8L} \mathbb{V}
\bigl( g(X_1) \bigr) - \frac{\beta}{n} \log\frac{1}{\pi}.
\end{eqnarray*}
Define $n$ i.i.d. random variables $Y_{1},\ldots,Y_{n}$ by
\begin{eqnarray*}
&& Y_{i} = g(X_i).
\end{eqnarray*}
Almost\vspace*{1pt} surely, $| Y_{i} | \le4 L_0$ and $| Y_{i} - \mathbb{E}Y_{i}|
\le8L_0$.
By applying \eqref{eqbernstein2} to $Y_{1},\ldots, Y_{n}$
with $b= 8L_0$ and $a=\frac{1}{8L}$,
we get that for any $t>0$ with $x=t+\log\frac{1}{\pi}$,
\begin{eqnarray*}
\mathbb{P} \biggl( (P-P_n)g - \frac{1}{8L} \mathbb{V} \bigl(
g(X_1) \bigr) > \frac{\beta x}{n} \biggr) &\le & \exp(-x)
\\
\mathbb{P} \biggl( \zeta_n > \frac{\beta x}{n} \biggr) \le\mathbb
{P} \biggl( (P-P_n)g - \frac {1}{8L} \mathbb{V} \bigl(
g(X_1) \bigr) - \frac{\beta}{n} \log\frac{1}{\pi} >
\frac {\beta t}{n} \biggr) &\le & \pi\exp(-t).
\end{eqnarray*}
\upqed\end{pf}

\subsection{Strong convexity}
\label{sproofs-H}

\begin{pf*}{Proof of Proposition~\protect\ref{propH}}
We will first prove that for any $\theta, \theta'$,
\begin{equation}
\label{eqstrong-convex-H} H_n(\theta) - H_n\bigl(
\theta'\bigr) = \bigl\langle\nabla H_n\bigl(
\theta'\bigr), \theta- \theta' \bigr\rangle +
\tfrac{1}{2} \llVert f_\theta- f_{\theta'} \rrVert
^2.
\end{equation}
Using the bias-variance decomposition of (\ref{eqbias-variance})
with $g=0$,
we get
\begin{eqnarray*}
&& \operatorname{pen}( \theta ) = \sum_{j=1}^M
\theta_j \llVert f_\theta- f_j \rrVert
^2 = - \llVert f_\theta \rrVert ^2 +\sum
_{j=1}^M\theta_j \llVert
f_j \rrVert ^2 .
\end{eqnarray*}
Thus, $H_n$ can be rewritten
as $H_n(\theta) = \frac{1}{2} \llVert  f_\theta \rrVert ^2 +
L(\theta)$
where $L$ is affine in $\theta$.
If we can prove $N(\theta) - N(\theta') = \langle\nabla N(\theta'),
\theta- \theta'\rangle+ \llVert  f_\theta- f_{\theta'} \rrVert ^2$
where $N(\theta) = \llVert  f_\theta \rrVert ^2$, then (\ref
{eqstrong-convex-H}) holds.
By simple properties of the norm,
\begin{eqnarray*}
\llVert f_\theta \rrVert ^2 - \llVert f_{\theta'}
\rrVert ^2 & =&  2 \int f_{\theta'} ( f_\theta-
f_{\theta'})\,d\mu+ \llVert f_\theta- f_{\theta'} \rrVert
^2
\\
& = & 2 {\theta'}^T G \bigl(\theta-
\theta'\bigr) + \llVert f_\theta- f_{\theta'} \rrVert
^2,
\end{eqnarray*}
where $G$ is the Gram matrix with elements $G_{j,k} =\int f_j f_k \,d\mu
$ for all $j,k=1,\ldots,M$.
The gradient at $\theta'$ of the function $\theta\rightarrow\llVert  f_\theta \rrVert ^2$
is exactly $2G\theta'$ so (\ref{eqstrong-convex-H}) holds.

The function $H_n$ is convex
and differentiable.
If ${\hat\theta}$ minimizes $H_n$ over the simplex,
then for any $\theta\in\Lambda^M$, $\langle\nabla H_n({\hat\theta}),
\theta- {\hat\theta}\rangle\ge0$ which proves (\ref{eqthat-H}).
\end{pf*}

\subsection{Tools for lower bounds}

\begin{prop}
\label{proprademacher}
There exists a countable set of functions $\varepsilon_1, \varepsilon_2,
\ldots$ defined on $[0,1]$ such that
for all $j,k>0$ with $k\ne j$,
\begin{eqnarray*}
&& \forall u\in[0,1),\qquad \varepsilon_j(u)  \in\{-1,1\},
\\
&&\int_{[0,1]} \varepsilon_j(x) \varepsilon
_k(x) \,dx  = 0,
\\
&& \int_{[0,1]} \varepsilon_j^2(x)
\,dx  = 1.
\end{eqnarray*}
Furthermore, if $U$ is uniformly distributed on $[0,1]$,
then $\varepsilon_1(U), \varepsilon_2(U),\ldots$ are i.i.d. Rademacher random
variables.
\end{prop}
See \cite{heil2011basis}, Definition~3.22, for an explicit construction
of these functions and a proof a their properties.

If $P\ll Q$ are two probability measures
defined on some measurable space, define their Kullback--Leibler
divergence and their $\chi_2$ divergence by
\begin{eqnarray*}
&& K(P,Q) = \int\log \biggl(\frac{dP}{dQ} \biggr) \,dP,\qquad
\chi_2(P,Q) = \int \biggl(\frac{dP}{dQ} - 1 \biggr)^2
\,dQ.
\end{eqnarray*}
The following comparison holds
\begin{eqnarray}\label{eqkullbackcompare}
&& K(P,Q)  \le\chi_2(P, Q).
\end{eqnarray}
Furthermore, if $n\ge1$ and $P^{\otimes n}$ denotes the $n$-product of
measures $P$,
\begin{eqnarray} \label{eqkullbackproduct}
&& K\bigl(P^{\otimes n},Q^{\otimes n}\bigr)  = n K(P, Q).
\end{eqnarray}
The proofs of \eqref{eqkullbackcompare} and \eqref{eqkullbackproduct}
are given in
\cite{tsybakov2009introduction}, Lemma~2.7 and page 85.

\begin{lemma}
\label{lemmaminimax}
Let $(\Omega, \mathcal{A})$ be a measurable space and $m\ge1$.
Let $m\ge1$ and $A_0,\ldots,A_m \in\mathcal{A}$ be disjoint events:
$A_j \cap A_k = \varnothing$ for any $j\ne k$.
Assume that $Q_0, \ldots, Q_m$ are probability measures on
$(\Omega, \mathcal{A})$ such that
\begin{eqnarray*}
&& \frac{1}{m} \sum_{j=1}^m
K(Q_j, Q_0) \le\chi< \infty.
\end{eqnarray*}
Then,
\begin{eqnarray*}
&& \max_{j=0,\ldots,m} Q_j(\Omega\setminus A_j)
\ge\frac{1}{12} \min \bigl(1, m\exp(- 3 \chi)\bigr).
\end{eqnarray*}
\end{lemma}
Lemma~\ref{lemmaminimax} can be found in \cite{kerkyacharian2013optimal}, Lemma~3.
It is a direct consequence of
\cite{tsybakov2009introduction}, Proposition~2.3, with $\tau^* = \min(m^{-1}, \exp(-3\chi))$.

\begin{cor}[(Minimax lower bounds)]
\label{corminimax}
Let $n\ge1$ be an integer and
$s>0$ be a positive number.
Let $m\ge1$ and $q_0,\ldots,q_m$ be a family of densities with respect
to the same measure $\mu$.
Assume that for any $j\ne k$,
\begin{eqnarray}
\label{eqdisjoint}
\Vert q_j - q_k \Vert^2
\ge4s > 0.
\end{eqnarray}
If $P_k^{\otimes n}$ denotes the product measure associated
with $n$ i.i.d. observations drawn from the density $q_k$,
assume that
\begin{eqnarray*}
&&\frac{1}{m} \sum_{j=1}^m K
\bigl(P_j^{\otimes n}, P_0^{\otimes n}\bigr) \le
\chi
\end{eqnarray*}
for some finite $\chi>0$.
Then, for any estimator $\hat T_n$,
\begin{eqnarray*}
&&\max_{k=0,\ldots,m} \mathbb{P}_k^{\otimes n} \bigl(
\Vert\hat T_n - q_k\Vert^2 \ge s \bigr) \ge
\frac{1}{12}\min\bigl(1,m\exp(-3\chi)\bigr).
\end{eqnarray*}
\end{cor}
\begin{pf}
For any estimator $\hat T_n$, for any $j=0,\ldots,m$ define
the events
\begin{eqnarray*}
&& A_j = \bigl\{ \llVert \hat T_n - q_j
\rrVert ^2 < s \bigr\}.
\end{eqnarray*}
These events are disjoint because of the triangle inequality
and (\ref{eqdisjoint}).
Applying Lemma~\ref{lemmaminimax} completes the proof.
\end{pf}

\subsection{Lower bound theorems}

\subsubsection{Lower bounds with exponential tails}

\begin{pf*}{Proof of Theorem~\protect\ref{thmlower-exp}}
Let $\varepsilon_2, \ldots, \varepsilon_M$ be $M-1$ functions
from Proposition~\ref{proprademacher}.
Consider the dictionary
$\{f_1,\ldots,f_{M}\}$ such that for all $(u_1,\ldots,u_d)\in\mathbf{R}^d$
\begin{eqnarray*}
&& f_1(u_1,\ldots,u_d) = \frac{L}{2}
\mathbf{1}_{[0,1]} \biggl(\frac
{L}{2} u_1 \biggr) \prod
_{q=2}^d \mathbf{1}_{[0,1]}(u_q),
\end{eqnarray*}
and for $j\ge2$
\begin{eqnarray*}
&& f_j(u_1,\ldots,u_d) = \frac{L}{2}
\biggl(1 + \sqrt{\frac{\log(M) +
x}{3n}} \varepsilon_j \biggl(
\frac{L}{2} u_1 \biggr) \biggr) \mathbf {1}_{[0,1]}
\biggl(\frac{L}{2} u_1 \biggr) \prod
_{q=2}^d \mathbf{1}_{[0,1]}(u_q).
\end{eqnarray*}
Since $\frac{\log M +x}{n} < 3$, these functions are densities and
satisfy $\llVert  f_j \rrVert _\infty < L$.

For any $j\ne k$,
\begin{eqnarray}
\label{eqlower-dist}
&& \llVert f_j - f_k \rrVert
^2 \ge L \frac{\log(M) + x}{6n}
\end{eqnarray}
and (\ref{eqlower-dist}) is true with equality when $j=1$.
If $P_k^{\otimes n}$ denotes the probability with respect to $n$ i.i.d.
random variables with density $f_j$, the properties \eqref
{eqkullbackcompare} and \eqref{eqkullbackproduct} give that for any
$k\ge2$,
\begin{eqnarray*}
K\bigl(P_k^{\otimes n},P_1^{\otimes n}\bigr) & =&
n K\bigl(P_k^{\otimes1},P_1^{\otimes1}\bigr)
\\
& \le & n \chi_2\bigl(P_k^{\otimes1},P_1^{\otimes1}
\bigr)
\\
& =&  n \frac{2}{L} \llVert f_k - f_1 \rrVert
^2
\\
& =& \frac{\log(M) + x}{3}.
\end{eqnarray*}
Applying Corollary~\ref{corminimax} with $m=M-1$ yields that for any estimator
$\hat T_n$,
\begin{eqnarray*}
\sup_{j=1,\ldots,M} P_j^{\otimes n} \biggl( \llVert
\hat T_n - f_j \rrVert ^2 \ge L
\frac{\log(M) + x}{24n} \biggr) & \ge & \frac{1}{12} \min\biggl(1,
\frac{M-1}{M} \exp(-x)\biggr)
\\
& \ge & \frac{1}{24} \exp(-x).
\end{eqnarray*}
\upqed\end{pf*}

\begin{pf*}{Proof of Theorem~\ref{thmselectors}}
\label{pselectors}
Let $\varepsilon_1,\ldots,\varepsilon_M$ be $M$ functions from Proposition~\ref
{proprademacher}.

For $(u_1,\ldots,u_d)\in\mathbf{R}^d$, we define a dictionary $\{
f_1,\ldots,f_M\}$ by
\begin{eqnarray*}
&& f_j(u_1, \ldots, u_d) = \frac{L}{2}
\biggl(1 + \varepsilon_j \biggl(\frac
{L}{2} u_1
\biggr) \biggr) \mathbf{1}_{[0,1]} \biggl(\frac{L}{2}
u_1 \biggr) \prod_{q=2}^d
\mathbf{1}_{[0,1]}(u_q),
\end{eqnarray*}
and we define $M$ densities $\{d_1,\ldots,d_M\}$:
\begin{eqnarray*}
&& d_j(u_1,\ldots,u_d) = \frac{L}{2}
\biggl(1 + \gamma\varepsilon_j \biggl(\frac{L}{2}
u_1 \biggr) \biggr) \mathbf{1}_{[0,1]} \biggl(
\frac
{L}{2} u_1 \biggr) \prod_{q=2}^d
\mathbf{1}_{[0,1]}(u_q),
\end{eqnarray*}
for some $\gamma\in(0,\frac{1}{2})$ that will be specified later.
Due to the properties of the $(\varepsilon_j)$, the following holds
for any $j\ne k$
\begin{eqnarray*}
\Vert f_k - d_j \Vert^2 & =&
\frac{L}{2}\bigl(1 + \gamma^2\bigr),
\\
\Vert f_j - d_j \Vert^2 & =&
\frac{L}{2}(1-\gamma)^2,
\\
\Vert d_j - d_k \Vert^2 & =&  L
\gamma^2.
\end{eqnarray*}
Thus if $\hat S_n$ is any selector taking values in the discrete set
$\{f_1,\ldots,f_M\}$:
\begin{equation}\label{eqexcessdj}
\Vert\hat S_n - d_j \Vert^2 - \inf
_{l=1,\ldots,M}\Vert f_l -d_j\Vert
^2 = \Vert\hat S_n - d_j \Vert^2
- \Vert f_j -d_j\Vert^2 = 2 L \gamma
\mathbf{1}_{\hat S_n\ne f_j}.
\end{equation}
Let $P_k^{\otimes n}$ be the product measure associated with $n$ i.i.d.
random variables drawn from the density~$d_k$.
Equation (\ref{eqexcessdj}) ensures that with probability $\mathbb
{P}_j^{\otimes n}(\hat S_n\ne f_j)$, the excess risk is $2 L\gamma$.

For any $k\ne1$, using \eqref{eqkullbackcompare} and \eqref
{eqkullbackproduct}, we obtain
\begin{eqnarray*}
K\bigl(P_k^{\otimes n},P_1^{\otimes n}\bigr) & =&
n K\bigl(P_k^{\otimes1},P_1^{\otimes1}\bigr)
\\
& \le& n \chi_2\bigl(P_k^{\otimes1},P_1^{\otimes1}
\bigr)
\\
& \le &\frac{4}{L} n \llVert d_k - d_1 \rrVert
^2
\\
& =&  4 n \gamma^2,
\end{eqnarray*}
where we used that $d_1(u_1,\ldots,u_d) \ge L/4$ almost surely on the
common support of $d_k$ and $d_1$.

Now we choose $\gamma= \frac{1}{2\sqrt{3}}\sqrt{\frac{x+\log
M}{n}}$ such that
$\forall k\ne1,K(P_k^{\otimes n},P_1^{\otimes n}) \le\frac{x + \log M}{3}$.
Let $\hat S_n$ be any estimator with values in the discrete set $\{
f_1,\ldots,f_M\}$.
For any $j=1,\ldots,M$, define the event $A_j = \{ \hat S_n = f_j \}$.
The events are disjoint if $f_j \ne f_k$ for all $j\ne k$ (if this is
not satisfied,
we can always remove the duplicates).
By applying Lemma~\ref{lemmaminimax} with $m = M-1$ and $\chi= \frac
{1}{3}(x+ \log M)$, we get
\begin{eqnarray*}
&&\max_{j=1,\ldots,M} \mathbb{P}_j^{\otimes n} ( \hat
S_n \ne f_j ) \ge\frac{M-1}{12M} \exp(-x).
\end{eqnarray*}
Since $(M-1)/M \ge1/2$,
\begin{eqnarray*}
\max_{j=1,\ldots,M} \mathbb{P}_j^{\otimes n} \Bigl(
\Vert\hat S_n - d_j \Vert^2 - \inf
_{l=1,\ldots,M}\Vert f_l -d_j
\Vert^2 > 2 L \gamma \Bigr) & \ge & \frac{M-1}{12M} \exp(-x)
\\
& \ge& \frac{1}{24} \exp(-x).
\end{eqnarray*}
\upqed\end{pf*}

\subsubsection{ERM over the convex hull}
\label{sproofhull}

\begin{pf*}{Proof of Proposition~\protect\ref{properm-hull}}
By homogeneity, it is enough to prove the case $L=2$.
Let $\phi_1,\ldots,\phi_M, \phi_{M+1}$ be $M+1$ functions
given by Proposition~\ref{proprademacher}.
Consider the probability density $f = \mathbf{1}_{[0,1]}$ and the
dictionary of $2M+1$ functions
\begin{eqnarray*}
&&\mathcal{D} = \{\mathbf{1}_{[0,1]} \} \cup \bigl\{(1 \pm
\phi_j \phi_{M+1})\mathbf{1}_{[0,1]}, j=1,\ldots,M
\bigr\}.
\end{eqnarray*}
The true density is in the dictionary thus $\min_{g\in\mathcal{D}}
\llVert  f - g \rrVert ^2 = 0$.
Also, all the elements of the dictionary are uniformly bounded by $L=2$.

The convex hull of the dictionary is the set
\begin{eqnarray*}
&& \bigl\{ g_\lambda= (1 + f_\lambda\phi_{M+1})
\mathbf{1}_{[0,1]},  \lambda\in\mathbf{R}^M, |
\lambda|_1 \le1 \bigr\},
\end{eqnarray*}
where
$f_\lambda= \sum_{j=1}^M \lambda_j \phi_j$ and $|\cdot|_1$ is the
$\ell_1$ norm in $\mathbf{R}^M$.

For all $\lambda\in\mathbf{R}^M$ with $|\lambda|_1 \le1$,
$\llVert  f - g_\lambda \rrVert ^2 = |\lambda|^2_2$
where $|\cdot|_2$ is the $\ell_2$ norm in $\mathbf{R}^M$.

Let $\mathcal{L}_\lambda:=\llVert  g_\lambda \rrVert
^2 - 2 g_\lambda+
2f - \llVert  f \rrVert ^2 = |\lambda|_2^2 - 2 f_\lambda\phi_{M+1}$.
Since the empirical process is linear in $\lambda$,
the proof from \cite{lecue2009aggregation} can be adapted as follows.
Given $n$ i.i.d. observations $X_1,\ldots,X_n$ generated by the density $f$,
\cite{lecue2009aggregation}, Lemma~5.4,
states that for every $r>0$, with probability greater than $1-6\exp
(-C_2 M)$,
\begin{eqnarray*}
&& c_0 \sqrt{\frac{r}{M}} \le c_1 \sqrt{
\frac{rM}{n}} \le \sup_{\lambda\in\mathbf{R}^M, |\lambda|_2 \le\sqrt{r} } P_n(f_\lambda
\phi_{M+1}) \le c_2 \sqrt{\frac{rM}{n}} \le
c_3 \sqrt{\frac{r}{M}},
\end{eqnarray*}
where $c_0,c_1, c_2, c_3>0$ are absolute constants.

Let $r\le1/M$ that will be specified later (such that if $|\lambda|_2
\le\sqrt{r}$ then $|\lambda|_1 \le1$).
On the one hand,
\begin{eqnarray*}
&& \inf_{\lambda\in\mathbf{R}^M, |\lambda|_2 \le\sqrt{r} } P_n \mathcal {L}_\lambda \le r
- 2 \sup_{\lambda\in\mathbf{R}^M, |\lambda|_2\le\sqrt{r}} P_n( f_\lambda
\phi_{M+1}).
\end{eqnarray*}
Given that $n \sim M^2$, using the above high probability estimate,
there exists a positive absolute constant $c_4$ such that
for all $r\le c_3^2/ (4M)$,
with probability greater than $1-6\exp(-C_2 M)$,
$
\inf_{\lambda\in\mathbf{R}^M, |\lambda|_2 \le\sqrt{r} } P_n
\mathcal
{L}_\lambda
\le
\sqrt{r}(\sqrt{r} - c_3/\sqrt{M})
\le
- c_4 \sqrt{r/M}
$, where\vspace*{1pt} $c_4=c_3/2$.

On the other hand, if $\rho\le1/M$, with probability greater than
$1-6\exp(-C_2 M)$,
\begin{eqnarray*}
&&\sup_{\lambda\in\mathbf{R}^M, |\lambda|_2 \le\sqrt{\rho} } \bigl| (P_n-P) \mathcal{L}_\lambda\bigr|
= 2 \sup_{\lambda\in\mathbf{R}^M, |\lambda|_2 \le\sqrt{\rho} } \bigl| (P_n-P) f_\lambda
\phi_{M+1} \bigr| \le 2 c_3 \sqrt{\frac{\rho}{M}}.
\end{eqnarray*}
Finally, choose $r,\rho$ such that $2 c_3 \sqrt{\rho/M} < c_4 \sqrt
{r/M}$ and $\rho> c_5/\sqrt{n}$ for some absolute constant $c_5>0$, then
with probability greater than $1-12\exp(-C_2 M)$,
\begin{eqnarray*}
&&\inf_{\lambda, |\lambda|_2\le\sqrt{\rho}} P_n\mathcal{L}_\lambda \ge -
\sup_{\lambda, |\lambda|_2 \le\sqrt{\rho} } \bigl| (P_n-P) \mathcal {L}_\lambda\bigr|
\ge - 2 c_3 \sqrt{\frac{\rho}{M}} > - c_4 \sqrt{
\frac{r}{M}} \ge \inf_{\lambda, |\lambda|_2\le\sqrt{r}} P_n
\mathcal{L}_\lambda.
\end{eqnarray*}
Thus with high probability,
$\inf_{\lambda, |\lambda|_2\le\sqrt{\rho}} P_n\mathcal
{L}_\lambda
>
\inf_{\lambda, |\lambda|_1\le1} P_n\mathcal{L}_\lambda$.
The inequality is strict so
the empirical risk minimizer has a risk greater than $\rho$.
As $\rho$ satisfies $\rho> C_3/\sqrt{n}$,
the proof is complete.
\end{pf*}

\subsubsection{Exponential weights}
\label{sproofew}

If $Y_1,\ldots,Y_m$ are i.i.d. with $\mathbb{P} ( Y_1= \pm1
 ) = 1/2$,
then for all $u\in[0,\sqrt{m}/4]$,
\begin{equation}\label{eqmatousek}
\tfrac{1}{15} \exp\bigl(-4 u^2\bigr) \le \mathbb{P} (
Y_1+ \cdots + Y_m \ge u \sqrt{m} ) \le \exp
\bigl(-u^2 / 2\bigr).
\end{equation}
A proof of the lower bound can be found in
\cite{matousek2008probabilistic}, Proposition~7.3.2,
and a standard Chernoff bound provides the upper bound.
The following proof uses arguments similar to \cite{dai2012deviation}.

\begin{pf*}{Proof of Proposition~\protect\ref{propew}}
By homogeneity, it is enough to prove the case $L=1$.
Let $\varepsilon_1,\varepsilon_2, \varepsilon_3$ be $3$ functions from Proposition~\ref
{proprademacher}.
Let $f=\mathbf{1}_{[0,1]}$ be the unknown density and let
\begin{eqnarray*}
f_1  &=&  f + \varepsilon_1, \qquad f_2 = f +
\biggl(1+\frac{1}{\sqrt n}\biggr) \varepsilon_2, \qquad f_3 =
f_2 + \frac{\alpha}{\sqrt n} \varepsilon_3,
\\
\pi_1 &=& 1/(8\sqrt{n}),  \qquad \pi_2 = 1/(8\sqrt{n}),\qquad
  \pi_3 = 1 - 1/(4\sqrt{n}),
\end{eqnarray*}
where $0 \le\alpha\le n^{1/4}$ will be specified later.
The best function in the dictionary is $f_1$: $\llVert  f_1-f \rrVert ^2 = \min_{j=1,\ldots,M}\llVert  f_j-f \rrVert ^2$.

Let $E$ be the event $\{R_n(f_2) + 2/\sqrt{n} \le R_n(f_1)\}$.
By simple algebra,
\begin{eqnarray*}
&& E = \bigl\{1 + 4\sqrt n - 2 \sqrt{n} P_n(\varepsilon_2)
\le2 n  P_n( \varepsilon_2 - \varepsilon_1) \bigr
\} \supseteq \bigl\{7 \sqrt{n} \le2 n  P_n( \varepsilon_2 -
\varepsilon_1) \bigr\},
\end{eqnarray*}
where for the inclusion we used $1\le\sqrt{n}$ and $|P_n(\varepsilon
_2)| \le1$.
The $2n$ random variables $(\varepsilon_j(X_i))_{j=1,2;  i=1,\ldots,n}$ are
i.i.d. Rademacher random variables, so applying the lower bound of~\eqref{eqmatousek} with $m=2n$ and $u=7\sqrt{2}/4$ yields $\mathbb
{P} ( E  ) \ge C_2 >0$ for some absolute constant $C_2$.
Now set
$\alpha^2 = 8 \log(2 n / C_2)$,
and choose $N_0$ such that for all $n\ge N_0$, $8 \log(2 n / C_2)> 0$ and
$\alpha^2 \le\sqrt{n}$.

Let $F :=\{ R_n(f_3) \le R_n(f_1) \}$ and define
\begin{eqnarray*}
&& G = \bigl\{ 2 (\alpha/\sqrt{n}) P_n(\varepsilon_3) \le
\alpha^2 /n - 2/\sqrt{n} \bigr\}.
\end{eqnarray*}
Since $R_n(f_3) = R_n(f_2) + \alpha^2 /n - 2 (\alpha/\sqrt{n})
P_n(\varepsilon_3)$,
we have
$E \cap G^c \subseteq F$.
As $\alpha^2 \le\sqrt{n}$ holds, we have
$\alpha^2 - 2 \sqrt{n} \le- \alpha^2$ and
\begin{eqnarray*}
&& G \subseteq \bigl\{ (2\alpha/n) P_n(\varepsilon_3)
\le- \alpha^2 / n \bigr\} = \bigl\{ - n P_n(
\varepsilon_3) \ge\sqrt{n} \alpha/ 2 \bigr\}.
\end{eqnarray*}
The random variable $ - n P_n(\varepsilon_j)$
is the sum of $n$ independent Rademacher random variables.
Applying the upper bound of \eqref{eqmatousek}
to $u=\alpha/2$, we have $\mathbb{P} ( G  ) \le\exp
(-\alpha^2 / 8) = C_2/(2n)$
since $\alpha= 8 \log(2 n / C_2)$.
Now as $F^c \subset E^c \cup G$,
\begin{eqnarray*}
&&\mathbb{P} \bigl( E^c \cup F^c \bigr) \le \mathbb{P}
\bigl( E^c \cup G \bigr) \le (1-C_2) + \frac{C_2}{2n}
\le1 -C_2/2 < 1.
\end{eqnarray*}
The probability of the event $E \cap F$ is greater than $C_0:= C_2/2$.
On this event,
$R_n(f_2)\le R_n(f_1)$ and
$R_n(f_3)\le R_n(f_1)$
thus
\begin{eqnarray*}
\hat\theta^{\mathrm{EW}, \beta}_1 & =& \frac{\pi_1 \exp(-R_n(f_1)/\beta)}{\pi_1 \exp(-R_n(f_1)/\beta
) + \pi_2 \exp(-R_n(f_2)/\beta) + \pi_3 \exp(-R_n(f_3)/\beta)}
\\
& \le & \frac{\pi_1 \exp(-R_n(f_1)/\beta)}{(\pi_1 + \pi_2 + \pi_3) \exp
(-R_n(f_1)/\beta)} = \pi_1 = \frac{1}{8\sqrt{n}} .
\end{eqnarray*}
Let $\theta_1 = \hat\theta^{\mathrm{EW}, \beta}_1$ for simplicity.
As $(\varepsilon_1,\varepsilon_2, \varepsilon_3)$ is an orthonormal system,
\begin{eqnarray*}
\llVert f_{\hat\theta^{\mathrm{EW}, \beta}} - f \rrVert ^2 - \llVert f_1
- f \rrVert ^2 & \ge& \bigl\llVert \theta_1
f_1 + (1-\theta_1)f_2 - f \bigr\rrVert
^2 - \llVert f_1 - f \rrVert ^2
\\
& =& (1-\theta_1)^2 \llVert f_2-f \rrVert
^2 - \bigl(1-\theta _1^2\bigr)\llVert
f_1-f \rrVert ^2
\\
& \ge & 2(1-\theta_1)^2 / \sqrt{n} + \bigl[(1-
\theta_1)^2 - \bigl(1-\theta _1^2
\bigr)\bigr]
\\
& \ge & 1 /(2\sqrt{n}) - 2 \theta_1
\\
& \ge & 1 /(2\sqrt{n}) - 2/(8\sqrt{n}) \ge1/ (4\sqrt{n}).
\end{eqnarray*}
\upqed\end{pf*}

The proof of Proposition~\ref{propew2} is based on estimates from \cite
{lecue2013optimality}
and highlights the similarities between regression with random design
and density estimation with the $L^2$ risk.

\begin{pf*}{Proof of Proposition~\protect\ref{propew2}}
By homogeneity, it is enough to prove the case $L=1$.
The strategy is to construct an example for density estimation
such that the calculations from \cite{lecue2013optimality}, Proof of Theorem~A,
can be leveraged.
Let $f_Y$ be the probability density
\begin{eqnarray*}
&& f_Y(x) = \cases{ 1/4 + 1/(2\sqrt{n}), & \quad$ \mbox{if } x\in
[-2,0)$, \vspace*{3pt}
\cr
1/4 - 1/(2\sqrt{n}),&\quad $ \mbox{if } x\in(0,2
]$,}
\end{eqnarray*}
and $0$ elsewhere.
Let $\{f_1 = \frac{1}{2} \mathbf{1}_{[-2,0)}, f_2 = \frac{1}{2}
\mathbf{1}_{[0,2)} \}$
be the dictionary.
Let
\begin{eqnarray*}
&& \mathcal{L}_2(y) :=\llVert f_2 \rrVert ^2 -
2 f_2(y) + 2 f_1(y) - \llVert f_1 \rrVert
^2,  \qquad\forall y\in\mathbf{R},
\end{eqnarray*}
and observe that $\mathcal{L}_2(Y) = -X$ where $X= \mathbf
{1}_{[0,2)}(Y) - \mathbf{1}_{[-2,0)}(Y)$
so that $X$ satisfies
\begin{eqnarray*}
&& X = %
\cases{ 1, &  \quad$\mbox{with probability } 1/2 - 1/\sqrt{n}$,
\vspace*{3pt}
\cr
-1, & \quad$\mbox{with probability } 1/2 + 1/\sqrt{n}$.}
\end{eqnarray*}
By definition of $\mathcal{L}_2$,
\begin{eqnarray*}
&& P \mathcal{L}_2 = \mathbb{E}\mathcal{L}_2(Y) = \llVert
f_2 - f_Y \rrVert ^2 - \llVert
f_1 - f_Y \rrVert ^2.
\end{eqnarray*}
As $P \mathcal{L}_2 = \mathbb{E}[-X] = 2/\sqrt{n} > 0$, $f_1$ is the best
function in the dictionary
and $P\mathcal{L}_2$ is the excess risk of $f_2$.
Finally, let
\begin{eqnarray*}
&&\alpha= \frac{\llVert  f_1 - f_2 \rrVert ^2}{P\mathcal{L}_2} = \frac{\sqrt{n}}{2}.
\end{eqnarray*}

For any $\theta\in[0,1]$, let $f_\theta= \theta f_1 + (1-\theta) f_2$.
An explicit calculation of the excess risk of $f_\theta$ yields
\begin{eqnarray*}
\llVert f_\theta- f_Y \rrVert ^2 - \llVert
f_1 - f_Y \rrVert ^2 & =&
\theta^2 \llVert f_1 \rrVert ^2 + (1-
\theta)^2 \llVert f_2 \rrVert ^2 - 2
\mathbb{E} \bigl[f_\theta(Y)\bigr] + 2 \mathbb{E}\bigl[f_1(Y)
\bigr] - \llVert f_1 \rrVert ^2
\\
& =& - \theta(1-\theta) \llVert f_1 - f_2 \rrVert
^2 + (1-\theta) \mathbb{E}[-X]
\\
& = &\bigl(1 - \theta- \theta(1-\theta) \alpha\bigr) P\mathcal{L}_2.
\end{eqnarray*}
Given $n$ independent observations $Y_1,\ldots,Y_n$ with common density $f$,
define $X_i= \mathbf{1}_{[0,2)}(Y_i) - \mathbf{1}_{[-2,0)}(Y_i)$ as above.
The exponential weights estimator with temperature $\beta$ can be
written as
\begin{eqnarray*}
&&\hat f^{\mathrm{EW}}_\beta= \hat\theta_1
f_1 + (1-\hat\theta_1) f_2, \qquad  \hat
\theta_1 :=\frac{1}{1 + \exp(- (n/\beta) ({1}/{n})
\sum_{i=1}^n [- X_i])},
\end{eqnarray*}
and its excess risk is $\llVert \hat f^{\mathrm{EW}}_\beta- f_Y \rrVert ^2 - \llVert  f_1 - f_Y \rrVert ^2 = (1 - \hat\theta_1
- \hat\theta_1(1-\hat\theta_1)\alpha)
P \mathcal{L}_2$.

The constants $\alpha$ and $P \mathcal{L}_2$,
the law of $X_1,\ldots,X_n,\hat\theta_1$
are the same as in \cite{lecue2013optimality}, Proof of Theorem~A,
thus the lower bounds in expectation and probability of the quantity
$(1 - \hat\theta_1 - \hat\theta_1(1-\hat\theta_1)\alpha)$ in
Lecu{\'e} and Mendelson
\cite{lecue2013optimality}
also hold here and yield the lower bound of Proposition~\ref{propew2}.
\end{pf*}

\section*{Acknowledgement}
This work was supported by GENES and by the grant Investissements d'Avenir
(ANR-11-IDEX-0003/Labex Ecodec/ANR-11-LABX-0047).

%





\printhistory
\end{document}